\documentclass[11pt]{amsart}

\usepackage{amsmath,amssymb,amsthm,mathtools}
\usepackage{enumitem}
\usepackage{tikz}
\usepackage[hidelinks]{hyperref}

\numberwithin{equation}{section}
\theoremstyle{plain}
\newtheorem{theorem}{Theorem}[section]
\newtheorem{corollary}[theorem]{Corollary}
\newtheorem{proposition}[theorem]{Proposition}
\newtheorem{lemma}[theorem]{Lemma}
\newtheorem{claim}[theorem]{Claim}
\newtheorem*{theoremA}{Theorem A}
\newtheorem*{theoremB}{Theorem B}
\newtheorem*{structurethm}{Structure Theorem}
\newtheorem*{theoremC}{Theorem C}
\newtheorem*{theoremD}{Theorem D}
\theoremstyle{definition}
\newtheorem{definition}[theorem]{Definition}
\newtheorem{remark}[theorem]{Remark}
\newtheorem{convention}[theorem]{Convention}

\DeclareMathOperator{\girth}{girth}
\DeclareMathOperator{\halfgirth}{half\text{-}girth}
\newcommand{\F}{\mathbb F}
\newcommand{\out}{\mathrm{out}}
\newcommand{\inn}{\mathrm{in}}
\newcommand{\one}{\mathbf 1}
\newcommand{\eps}{\varepsilon}
\newcommand{\abs}[1]{\lvert #1\rvert}
\newcommand{\defect}{D}

\title[Global girth obstruction for taiko product structures]{A global girth obstruction for Garg--Mineyev taiko product structures}

\author{Henry Shin}
\address{San Diego, CA, USA}
\email{hkshin@gmail.com}
\hypersetup{
  pdftitle={A global girth obstruction for Garg--Mineyev taiko product structures},
  pdfauthor={Henry Shin},
  pdfkeywords={Kaplansky conjectures; group rings; torsion-free CAT(0) groups; support size; taikos; incidence geometry}
}

\begin{document}

\begin{abstract}
Mineyev's taiko construction, in Garg--Mineyev's finite support-size formulation, gives a
concrete route from finite support data to zero divisors and units in group rings of torsion-free
$\mathrm{CAT}(0)$ groups over $\F_2$. We prove that this triple-girth product-structure route is
globally closed. No product structure, even or odd, with support sizes $m,n\ge2$ admits a coherent
orientation for which no-fold and the triple-girth condition hold. Consequently the
Garg--Mineyev triple-girth product-structure assembly route produces neither zero-divisor nor unit
counterexamples over $\F_2$ for any such support-size pair.

The obstruction is structural, not a bounded-search artifact. High middle-link girth forces signed
colors into a balanced near-disjoint rectangle decomposition of the board, with the single odd
defect omitted. The product identity, pressure inequalities, Fisher inequalities, and a dual
Fisher bound force the middle link to have girth $4$ or $6$; in the girth-six case, the minimum of
the two horizontal-link girths is at most $5$. This dichotomy rules out every triple-girth branch.
A weighted dual Fisher inequality and an exact finite certificate sharpen the frontier: if the
middle link has girth $6$, the horizontal girth is at most $4$, and characteristic-two affine-plane
constructions attain equality. Thus the Garg--Mineyev finite failures reflect a structural barrier
in the taiko geometry itself. The finite certificate is used only for this sharper frontier, not
for the no-$\mathsf T_4$ obstruction.
\end{abstract}

\maketitle

\section{Introduction}

Kaplansky's zero-divisor conjecture predicts that if $K$ is a field and $G$ is a torsion-free
group, then the group ring $K[G]$ has no nontrivial zero divisors. It belongs to the classical
Kaplansky problems on group rings of torsion-free groups, with antecedents in Higman's work on
units \cite{Higman1940,Kaplansky1957,Kaplansky1970} (for background on group rings see
\cite{Passman}). The related unit conjecture is now known to be false over many coefficient fields: Gardam gave a concrete
counterexample over $\F_2$ \cite{GardamUnit}, Murray extended the phenomenon to every prime
characteristic \cite{Murray}, and Gardam later produced nontrivial units in complex group rings,
disproving the characteristic-zero case \cite{GardamComplex}. The zero-divisor conjecture remains
open. These are finite-support problems at their core: a zero-divisor relation or a nontrivial unit
is witnessed by finitely many group elements, and the way those supports multiply is a finite
combinatorial pattern inside the group.

The classical unique-product route makes this finite-support viewpoint explicit. In a
unique-product group, the product of two nonempty finite sets has a uniquely represented element,
which rules out the cancellations needed for zero divisors over domains. Torsion-free groups need
not have unique product, as shown by Rips--Segev and by later concrete families of Promislow and
Carter \cite{RipsSegev,Promislow,Carter}; beyond the unique-product regime, finite support
patterns become subtle rather than automatically forced apart. For torsion-free
$\mathrm{CAT}(0)$ groups, Mineyev's \emph{taiko}, or \emph{product-structure}, construction gives
one of the few explicit ways to turn such finite support data into a geometrically controlled
group-ring construction \cite{MineyevOrigami,GargMineyev}.

In the quantitative form developed by Garg and Mineyev, the two support sizes are encoded by a
finite product structure $\Pi=(A,B,P)$ of type $(m,n)$, with $|A|=m$ and $|B|=n$.  When the taiko
conditions hold, the assembly theorem gives a zero divisor over $\F_2$ in the even case and a
nontrivial unit in the odd case \cite[Theorem~16]{GargMineyev}.  The relevant hypotheses are
orientability $\mathsf T_1$, the no-fold condition $\mathsf T_2$, and the triple-girth condition
$\mathsf T_4$.  The latter requires the two
horizontal graphs $L_A,L_B$ (with $L_{AB}=L_A\sqcup L_B$) and a middle link $L_1$ to have high
girth simultaneously, in one of three calibrated regimes. For any graph $G$ set
\[
        \halfgirth(G):=\frac{\girth(G)}2,
\]
with the convention $\infty/2=\infty$; equivalently, $\halfgirth(G)\ge q$ means
$\girth(G)\ge 2q$. Thus the calibrated regimes are
\[
        (p,q)\in\{(6,3),\,(4,4),\,(3,6)\},
        \qquad \girth(L_{AB})\ge p,\quad \halfgirth(L_1)\ge q.
\]
Explicitly, the three regimes are
\[
\begin{aligned}
(6,3):&\quad \girth(L_{AB})\ge 6,\ \girth(L_1)\ge 6;\\
(4,4):&\quad \girth(L_{AB})\ge 4,\ \girth(L_1)\ge 8;\\
(3,6):&\quad \girth(L_{AB})\ge 3,\ \girth(L_1)\ge 12 .
\end{aligned}
\]
The additional condition $\mathsf T_3$ plays no role in a nonexistence theorem: any structure
satisfying all four taiko conditions satisfies the three we obstruct, and, under the same simple
middle-link convention used in the taiko conditions, a repeated pattern is a $4$-cycle in the
middle link \cite[\S2.9]{GargMineyev}.

Garg and Mineyev proved that every orientable product structure has $\halfgirth(L_1)\le4$
\cite[Theorem 1]{GargMineyev}, and by exhaustive search they excluded all types with one support
size at most five, all pairs in the square $1\le m,n\le13$, and further finite ranges
\cite[Theorems 2 and 3]{GargMineyev}.  These exclusions continue the small-support and
combinatorial-encoding tradition for group rings, exemplified by Schweitzer's work on zero
divisors with small support and matched-rectangle methods \cite{Schweitzer}. The question left
open is whether these finite failures are artifacts of bounded search, or whether the taiko
geometry itself contains a global obstruction to the triple-girth route.

We prove that the obstruction is global: the finite searches were detecting a structural
incompatibility, not merely running out of room.

\begin{theorem}\label{thm:main}
No product structure $\Pi=(A,B,P)$, even or odd, with $\abs A,\abs B\ge 2$ admits a coherent
orientation $O$ such that $(\Pi,O)$ is no-fold and satisfies the triple-girth condition
$\mathsf T_4$.
\end{theorem}

\begin{corollary}\label{cor:route}
Under the hypotheses of Garg--Mineyev's Theorem~16, as identified in Section~\ref{sec:taiko}, no product
structure with both support sizes at least two satisfying $\mathsf T_1$, $\mathsf T_2$, and
$\mathsf T_4$ can feed the taiko assembly. Consequently that triple-girth product-structure
assembly route produces neither a zero-divisor counterexample \textup{(even case)} nor a unit
counterexample \textup{(odd case)} to the Kaplansky conjectures over $\F_2$, for any such
support-size pair.
\end{corollary}

Thus the finite failures found by Garg and Mineyev are not small-support accidents: the
triple-girth condition is incompatible with orientability and no-fold in the taiko framework
itself.  The result is deliberately scoped. It is not a proof of the zero-divisor conjecture and
not a contradiction to the known unit counterexamples of Gardam and Murray; it rules out the
$\mathsf T_1,\mathsf T_2,\mathsf T_4$ taiko mechanism, for every support-size pair with both coordinates at least two, in both the even and
odd forms of the construction.  The odd half is especially telling for the group-ring story:
nontrivial units over $\F_2$ do exist, but none can arise from this triple-girth taiko route.

The proof is structural rather than an enlargement of the finite search.  The middle link is not
a technical auxiliary graph; it is the place where the support pattern rigidifies. High middle-link
girth forces the signed colors to form a near-disjoint rectangle incidence geometry on the board.
That geometry supplies an exact product identity, Fisher-type inequalities, and pressure
inequalities.  These constraints yield a dichotomy for the middle link, kill all three regimes of
$\mathsf T_4$, and at the same time explain why the obstruction is sharp rather than a blunt
counting artifact: affine planes realize the boundary case, while a final finite profile
certificate only closes the last $(5,3)$ frontier.

\subsection*{Results and proof architecture}
The proof of Theorem~\ref{thm:main} starts from the middle link. The two girth hypotheses act on
different layers of the object, but once one passes to the middle link the picture becomes
incidence-geometric. Throughout, for a signed color $s$ we
write $A_s$ and $B_s$ for the vertices of $A$ and $B$ carrying $s$, and write $s\ni v$ for
$v$ carrying $s$. The defect parameter is
$\delta\in\{0,1\}$: in the even case $\delta=0$ there is no defect; in the odd case $\delta=1$
the unique defect is a single board position $(a_0,b_0)$. We write $D_A\subseteq A$ and
$D_B\subseteq B$ for its two projections, so that $D_A=D_B=\varnothing$ when $\delta=0$ and
$D_A=\{a_0\}$, $D_B=\{b_0\}$ when $\delta=1$; thus the indicators $\one_{a\in D_A}$ and
$\one_{b\in D_B}$ are well defined in both cases and vanish identically when $\delta=0$.
In these introductory summaries, ``even or odd in the standing taiko sense'' means exactly this
size-two-cell situation, with at most one defect and $mn\equiv\delta\pmod2$; the formal standing
convention is fixed in Section~\ref{sec:taiko}.

\begin{structurethm}[Theorem~\ref{thm:structure}]
Let $\Pi$ be an orientable no-fold product structure, even or odd in the standing taiko sense, with
$\abs A,\abs B\ge 2$. Then, because the fixed no-fold orientation makes $L_1$ the simple bipartite incidence graph of
Lemma~\ref{lem:collapsed-middle}, $\girth(L_1)\ge 6$ \emph{if and only if} the near-disjointness condition
\textup{(ii)} below holds. Under the standing hypotheses, this near-disjointness together with
Lemma~\ref{lem:cross} makes $L_1$ the Levi graph of a partial linear space (with point set
$A\sqcup B$, line set the signed colors, and incidence $v\in A_s\cup B_s$): the rectangles
$A_s\times B_s$, indexed by the signed colors $s$, partition
$(A\times B)\setminus\defect$, and:
\begin{enumerate}[label=\textup{(\roman*)},leftmargin=2.6em]
\item \textup{(Tiling)} for each $a\in A$, $\ \sum_{s\ni a}\abs{B_s}=n-\one_{a\in D_A}$,
and dually for each $b\in B$;
\item \textup{(Near-disjointness)} for distinct $s\ne t$, $\ \abs{A_s\cap A_t}+\abs{B_s\cap B_t}\le 1$;
\item \textup{(Balance)} the signed colors pair as $(c,\out),(c,\inn)$ with
$\abs{A_{(c,\out)}}=\abs{A_{(c,\inn)}}=:x_c$ and $\abs{B_{(c,\out)}}=\abs{B_{(c,\inn)}}=:y_c$;
\end{enumerate}
and consequently
\begin{enumerate}[label=\textup{(\roman*)},leftmargin=2.6em,start=4]
\item \textup{(Product identity)} $\ \sum_s\abs{A_s}\abs{B_s}=mn-\delta$, that is
$2\sum_c x_cy_c=mn-\delta$;
\item \textup{(Half-size)} $\ x_c\le\lfloor\tfrac{m+1}2\rfloor$ and $y_c\le\lfloor\tfrac{n+1}2\rfloor$;
\item \textup{(Pressure)} $\ y_c(m-x_c)\le 2e_A-x_c+\delta$ and $x_c(n-y_c)\le 2e_B-y_c+\delta$,
where $e_A=\abs{E(L_A)}$, $e_B=\abs{E(L_B)}$;
\item \textup{(Fisher bounds)} $\ \displaystyle\sum_c x_c(x_c-1)\le\binom m2$ and
$\displaystyle\sum_c y_c(y_c-1)\le\binom n2$;
\item \textup{(Dual Fisher)} $\ \displaystyle\sum_{a\in A}\binom{d_A(a)}2+\sum_{b\in B}\binom{d_B(b)}2\le\binom r2$,
where $d_A(a)=\deg_{L_A}(a)$, $d_B(b)=\deg_{L_B}(b)$, and $r$ is the number of signed colors.
\end{enumerate}
\end{structurethm}

Thus the middle link records a purely incidence-geometric object, and $\girth(L_1)\ge 6$ is
\emph{exactly} the linearity of that object (Figure~\ref{fig:rectangles}). Equivalently, the
rectangles $A_s\times B_s$ form
a \emph{balanced near-disjoint biclique partition} of the complete bipartite graph $K_{m,n}$
minus the defect edge: a decomposition of its edges into complete bipartite subgraphs (the
signed colors), invariant under the involution $(c,\out)\leftrightarrow(c,\inn)$ with matching
part-sizes, in which any two parts meet in at most one vertex. Biclique partitions of complete
and complete bipartite graphs are classical \cite{GrahamPollak}; here near-disjointness is the
partial-linear-space condition, and the absence of a $4$-cycle in $L_1$ is the
Zarankiewicz--K\H{o}v\'ari--S\'os--Tur\'an condition \cite{KovariSosTuran} that no two lines
meet twice. A line may carry several points of one class --- $\abs{A_s}>1$ or $\abs{B_s}>1$,
the generic situation; what is forbidden is two points sharing two lines. The numerical
consequences (iv)--(vi) drive everything that follows, and the condition $\mathsf T_4$ now
bifurcates by which layer carries the extra girth.

\begin{theoremA}[high middle-link girth; Theorem~\ref{thm:A}]
No orientable no-fold product structure, even or odd in the standing taiko sense, with
$\abs A,\abs B\ge 2$ and $\girth(L_1)\ge 8$ exists. Equivalently, \emph{every} such orientable
no-fold product structure has $\girth(L_1)\le 6$. In particular both the $(4,4)$ and $(3,6)$ branches of $\mathsf T_4$ are
empty.
\end{theoremA}

\noindent
In the language of half-girth this reads $\halfgirth(L_1)\le 3$. Garg and Mineyev prove
$\halfgirth(L_1)\le 4$ for every orientable product structure \cite[Theorem 1]{GargMineyev},
without a no-fold hypothesis; our bound is sharper but assumes no-fold, which is part of the
taiko input $\mathsf T_1,\mathsf T_2,\mathsf T_4$ we are obstructing. The improvement is exactly what eliminates the $(4,4)$
branch, which the bound $\halfgirth(L_1)\le 4$ leaves open. It is optimal: $\halfgirth(L_1)=3$
is attained (Remark~\ref{rem:example}). The improvement is not formal --- it uses the
full-board tiling identity (iv), which is unavailable to the subpartition method of
\cite{GargMineyev} --- and our proof of Theorem~A is independent of their Theorem~1.

\begin{theoremB}[the $(6,3)$ branch; Theorem~\ref{thm:B}]
No orientable no-fold product structure, even or odd in the standing taiko sense, with
$\abs A,\abs B\ge 2$ and $\girth(L_A),\girth(L_B),\girth(L_1)\ge 6$ exists. Hence the
$(6,3)$ branch of $\mathsf T_4$ is empty.
\end{theoremB}

Theorem~A makes no hypothesis on the horizontal graphs: it is a bound on the middle link of
\emph{any} such orientable no-fold product structure. The single shape $(3,3)$, which carries the
only special case, is dispatched not by a horizontal-girth assumption but by the observation
that the relevant middle link is forced to be the subdivision of $K_4$
(Lemma~\ref{lem:threethree}). Theorem~B, in turn, is what makes the bound of Theorem~A sharp
from one side: it says the value $\girth(L_1)=6$ can occur only when the horizontal graphs are
themselves of low girth. Combining the two yields a clean dichotomy, which we take to be the
structural heart of the paper.

\begin{corollary}[middle-link dichotomy; Corollary~\ref{cor:dichotomy}]
Let $\Pi$ be an orientable no-fold product structure, even or odd in the standing taiko sense,
with $\abs A,\abs B\ge 2$. With the convention that forests have girth $\infty$,
$\girth(L_1)\in\{4,6\}$. Moreover if $\girth(L_1)=6$ then $\girth(L_{AB})\le 5$; equivalently,
$\girth(L_{AB})\ge 6$ forces $\girth(L_1)=4$.
\end{corollary}

The three branches of $\mathsf T_4$ now fall out as a single corollary of the dichotomy. The
$(6,3)$ branch demands $\halfgirth(L_1)\ge 3$, i.e.\ $\girth(L_1)\ge 6$, which by the dichotomy
forces $\girth(L_1)=6$ and hence $\girth(L_{AB})\le 5$ --- contradicting its own demand
$\girth(L_{AB})\ge 6$. The $(4,4)$ and $(3,6)$ branches demand $\girth(L_1)\ge 8$, contradicting
the dichotomy's $\girth(L_1)\le 6$ directly. Both values of the dichotomy occur, so it is sharp
and non-vacuous: an explicit $(3,3)$ structure realizes $\girth(L_1)=4$, while Garg and
Mineyev's Example~10 --- a full $(4,4)$ orientable no-fold structure with $\girth(L_{AB})=3$ ---
realizes $\girth(L_1)=6$ (Remark~\ref{rem:example}). In particular the bound
$\halfgirth(L_1)\le 3$ of Theorem~A is optimal, and the no-repeated-pattern condition
$\mathsf T_3$ is satisfiable; what fails is only its conjunction with high horizontal girth.

The dichotomy's horizontal bound $\girth(L_{AB})\le 5$ is itself nearly sharp: the value
$\girth(L_1)=6$ persists even when the horizontal graphs are made triangle-free, as an infinite
affine-plane family shows.

\begin{theoremC}[affine-plane sharpness; Theorem~\ref{thm:affine}]
For every power of two $q\ge 4$ there is an even orientable no-fold product structure of type
$(2q,2q)$ with
\[
        \girth(L_1)=6,\qquad L_A\cong L_B\cong K_{q,q},
\]
so $\halfgirth(L_1)=3$ and $\girth(L_{AB})=4$. Consequently the pair
$(\girth(L_{AB}),\halfgirth(L_1))=(4,3)$ is realized, and the horizontal frontier
\[
\begin{aligned}
        p^\ast=\min\{\,p\in\mathbb Z_{\ge3}:&\ \nexists\,\Pi\text{ orientable no-fold, in the standing sense, with}\\
        &\ \girth(L_{AB}(\Pi))\ge p\text{ and }\halfgirth(L_1(\Pi))\ge 3\,\}
\end{aligned}
\]
satisfies $p^\ast\in\{5,6\}$.
\end{theoremC}

In this construction the middle link is the point--line incidence graph of $\mathrm{AG}(2,q)$
restricted to four parallel classes, the signed-color rectangles form a balanced biclique
partition of $K_{2q,2q}$, and each horizontal graph is the complete bipartite graph between its
two classes. It is the matching lower construction to the obstruction: the dichotomy forbids
$\girth(L_{AB})\ge 6$ at $\halfgirth(L_1)=3$, and affine planes reach $\girth(L_{AB})=4$. Whether
girth five can be inserted between them is settled by the next theorem.

\begin{theoremD}[computer-assisted closure of the $(5,3)$ frontier; Theorems~\ref{thm:asymp53} and~\ref{thm:p5}]
No orientable no-fold product structure, even or odd in the standing taiko sense, has
$\girth(L_1)\ge 6$ together with $\girth(L_A),\girth(L_B)\ge 5$. Hence
$\girth(L_1)=6\Rightarrow\girth(L_{AB})\le 4$, and the horizontal frontier above satisfies
$p^\ast=5$. The unbounded part is excluded by a weighted dual Fisher inequality with a
near-Moore reduction; the bounded remainder reduces to twenty-four explicit small types, each
ruled out by a finite check over signed-color profiles.
\end{theoremD}

Together, Theorems~A--D determine the middle-link girth of every orientable no-fold product
structure in the standing even/odd taiko sense and pin the horizontal frontier at
$\halfgirth(L_1)=3$ exactly: the realizable pairs
$(\girth(L_{AB}),\halfgirth(L_1))$ at half-girth three are $(3,3)$ and $(4,3)$, and nothing
higher.

\subsection*{Two structural remarks}
The Moore bound we use is the Alon--Hoory--Linial bound \cite{AlonHooryLinial} for irregular
graphs of girth six, and it is attained at the incidence graphs of projective planes: a plane
of order $q$ has Levi graph of degree $q+1$ on $2(q^2+q+1)$ vertices, meeting the bound with
equality. Theorem~B may thus be read as saying that the taiko incidence geometry cannot reach
the projective-plane Moore extremal regime: in the $(6,3)$ branch the pressure inequalities push
the horizontal graphs toward the equality case of the Moore bound, typified by projective-plane
Levi graphs, but the product identity prevents the structure from sitting at that extreme. The
collapsed regime of Theorem~A is the corresponding degeneracy in which every line carries at
most one point of each class.

Second, the pressure inequalities (vi) are a logical necessity, not a convenience. We show in
Remark~\ref{rem:necessary} that the product identity (iv) together with the half-size bound
(v) and the Moore bounds $e_A\le E_m$, $e_B\le E_n$ is, by itself, simultaneously satisfiable;
indeed a single Cauchy--Schwarz estimate is exactly tight on the diagonal. What forbids the
construction is the finer information in (vi): no color can be heavy on both sides at once.

\subsection*{Scope}
The group-ring consequence should be read exactly, and no more broadly. The combinatorial
obstruction is coefficient-independent at the level of the taiko data: no coefficient field appears
in the statement of the dichotomy or of Theorem~\ref{thm:main}, so the same orientable no-fold
$\mathsf T_4$ product-structure data are impossible independently of coefficients. Where the
Garg--Mineyev assembly applies, this closes that taiko route; it does not rule out
coefficient-sensitive constructions of a different shape, and ``$\F_2$'' marks where Mineyev's
particular size-two-cell construction is currently applied (Remark~\ref{rem:rings}). Orientability
is essential: it is what makes the signed colors at the row-vertex and the column-vertex of a
covered pair overlap (Lemma~\ref{lem:cross}), hence what produces the rectangle partition. The odd
case is the more striking half: unit counterexamples are now known over $\F_2$, in every
prime characteristic, and in characteristic zero \cite{GardamUnit,Murray,GardamComplex}, yet no
such example arises from this route, since no odd taiko attains $\mathsf T_4$. Finally, the
dichotomy and the affine construction locate exactly where the method's
reach ends: the obstruction closes the pair $(6,3)$, Theorem~\ref{thm:affine} realizes $(4,3)$, and
Theorem~\ref{thm:p5} closes the intermediate $(5,3)$ (Remark~\ref{rem:frontier}), so the precise
girth frontier at $\halfgirth(L_1)=3$ is exactly $p^\ast=5$.

\subsection*{Organization}
Section~\ref{sec:taiko} recalls product structures and signed colors and records the
oriented-simple property underlying all degree counts. Section~\ref{sec:structure} develops
the incidence picture and proves the Structure Theorem. Section~\ref{sec:small} disposes of
small supports, in particular forcing $\abs A,\abs B\ge 6$ in the $(6,3)$ branch.
Sections~\ref{sec:A} and~\ref{sec:B} prove Theorems~A and~B. Section~\ref{sec:synthesis}
assembles the main results and records the sharpness and necessity statements.
Section~\ref{sec:affine} gives the affine-plane construction, realizing $(4,3)$, and then closes
the frontier: the near-Moore reduction and weighted dual Fisher inequality exclude unbounded
$(5,3)$ families, and an effective form with a finite profile certificate excludes the bounded
remainder, yielding $p^\ast=5$.

\section{Product structures and signed colors}\label{sec:taiko}

We recall the data, following \cite{GargMineyev}, in the form we use; for the passage to group
rings see \cite{MineyevOrigami,GargMineyev}. All references to Garg--Mineyev definitions and
Theorem~16 below are to the version \cite{GargMineyev}, namely arXiv:2501.07646v2.

The dictionary with the Garg--Mineyev notation is made explicit in the following table. In the
displayed formulas below we use $\mathsf T_i$ only as a label for the corresponding taiko condition,
not as additional structure beyond these definitions.
\begin{center}
\small
\begin{tabular}{@{}lll@{}}
Garg--Mineyev datum & notation here & role in this paper \\
\hline
support sizes & $(m,n)$ & $|A|=m$, $|B|=n$ \\
product structure & $\Pi=(A,B,P)$ & finite taiko data \\
$\mathsf T_1$ & orientability & cell-compatible orientations \\
$\mathsf T_2$ & no-fold & signed-color degree control \\
horizontal links & $L_A,L_B$ & collapsed colored graphs \\
middle link & $L_1$ & signed-color incidence graph \\
$\mathsf T_4$ & triple-girth alternatives & $(6,3),(4,4),(3,6)$ regimes
\end{tabular}
\end{center}

\begin{definition}
A \emph{product structure} of type $(m,n)$ is a triple $\Pi=(A,B,P)$ with $\abs A=m$,
$\abs B=n$, and $P$ a partition of the board $A\times B$ into \emph{cells}, satisfying the
disjoint-vertex condition: if two distinct positions $(a,b),(a',b')$ lie in the same cell,
then $a\ne a'$ and $b\ne b'$. The structure is \emph{even} if every cell has size $2$, and
\emph{odd} if $P$ consists of $2$-cells and exactly one $1$-cell, denoted $(a_0,b_0)$. Write
$\defect$ for the set of defect positions ($\defect=\varnothing$, $\delta:=0$, even;
$\defect=\{(a_0,b_0)\}$, $\delta:=1$, odd), so $mn\equiv\delta\pmod 2$. We write
$D_A\subseteq A$ and $D_B\subseteq B$ for the projections of $\defect$; thus $D_A=D_B=
\varnothing$ in the even case and $D_A=\{a_0\}$, $D_B=\{b_0\}$ in the odd case.
\end{definition}

Throughout the paper we assume $\abs A=m\ge 2$ and $\abs B=n\ge 2$; the boundary cases $m=1$ or
$n=1$ are degenerate for the taiko data, since then no cell can satisfy the disjoint-vertex
condition with a horizontal edge and $L_{AB}$ has no edges at all. We exclude this boundary
throughout: every statement below about ``an orientable no-fold product structure'' carries this
standing hypothesis. \emph{Moreover, unless stated otherwise every product structure is assumed even
or odd}: thus $\delta\in\{0,1\}$, every non-defect position lies in a $2$-cell, there is at most one
defect, and $mn\equiv\delta\pmod 2$. This is exactly the class produced by the taiko assembly of
\cite{GargMineyev}---the even case for a putative zero divisor, the odd case for a putative
unit---and the parity relation, the product identity, and the Fisher and pressure relations used
below are statements about this class. We call this the \emph{standing taiko sense}. Accordingly,
later theorem statements may either rely on this standing convention or repeat ``even or odd'' for
emphasis and standalone readability. The phrase ``orientable no-fold product structure'' is also
orientation-relative in the sense made precise in Definition~\ref{def:oriented-no-fold}: it denotes
a product structure together with, or admitting, a coherent orientation for which the no-fold
condition holds.

A $2$-cell $\{(a,b),(a',b')\}$ induces a horizontal-edge \emph{occurrence}
$\{a,a'\}$ on the $A$-side and a horizontal-edge occurrence $\{b,b'\}$ on the $B$-side. The
occurrence language is used only to state explicitly the quotient convention for the horizontal
links. Colors are the connected components of the occurrence graph whose vertices are all
horizontal-edge occurrences and whose edges are generated by
\begin{enumerate}[label=\textup{(\alph*)},leftmargin=2.6em,itemsep=1pt]
\item the two horizontal occurrences belonging to the same $2$-cell;
\item two occurrences, on the same side, inducing the same unordered pair of vertices.
\end{enumerate}
The horizontal links used in the taiko conditions are the Garg--Mineyev quotient links. In their
notation,
\[
\begin{aligned}
\bar E_A&=\bigl\{\{a,a'\}:\exists b,b'\text{ with }\{(a,b),(a',b')\}\in P\bigr\},\\
\bar E_B&=\bigl\{\{b,b'\}:\exists a,a'\text{ with }\{(a,b),(a',b')\}\in P\bigr\},
\end{aligned}
\]
so repeated unordered endpoint pairs are already identified, and $\bar E_{AB}=\bar E_A\sqcup
\bar E_B$. The directed set $E_{AB}$ is the directed double cover of this collapsed set. Their color map is
constant on the equivalence classes generated by cell pairings after this quotient, and their
orientation is a section $\bar E_{AB}\to E_{AB}$ satisfying the cell-compatibility rule below
\cite[Definitions~5--6, \S\S2.4, 2.6]{GargMineyev}. Thus all horizontal occurrences with the same
unordered endpoint pair are first identified as a single simple horizontal edge. The horizontal
graphs $L_A,L_B$ are the simple graphs obtained after this collapse, and a collapsed edge inherits
one component color. Once a coherent orientation is fixed, the collapsed edge also has a single
chosen direction. The no-fold condition is evaluated on these collapsed simple edges, and the
middle link is formed from their directed colors as in \cite[Definition~11]{GargMineyev}.

This quotient convention is the convention by which the term \emph{product structure} is used
below, unless an occurrence-level statement is explicitly being made. We write
$L_{AB}=L_A\sqcup L_B$. Different cells may induce the same simple graph edge; all quantities
$x_c,y_c,e_A,e_B$ below count distinct collapsed graph edges, not raw cell occurrences. From this
point on, unless ``occurrence'' is explicitly written, an edge of a horizontal graph means a
collapsed simple horizontal edge. This collapsed-edge interpretation is not an additional
hypothesis on $\Pi$; it is the imported Garg--Mineyev definition of the horizontal links, colors,
orientations, signed colors, middle link, and no-fold condition. Equivalently, repeated
cell-induced occurrences may exist on the board, but $\mathsf T_2$, $\mathsf T_4$, and $L_1$ see
only the single collapsed colored edge with its coherent direction. Girth of these directed graphs
means girth of the underlying undirected graph \cite[Remark~15]{GargMineyev}. For every graph in
this paper, the girth of a disconnected graph is the minimum of the girths of its components, and
an acyclic graph has girth $\infty$.

\begin{definition}[orientability and no-fold]\label{def:oriented-no-fold}
An orientation $O$ is a choice of one direction for each collapsed simple horizontal edge. A
\emph{coherent orientation} is such a choice satisfying the cell-compatibility rule: in every
$2$-cell $\{(a,b),(a',b')\}$, the $A$-edge occurrence and the $B$-edge occurrence are oriented
compatibly, either $a\to a'$ together with $b\to b'$, or $a'\to a$ together with $b'\to b$.
The product structure $\Pi$ is \emph{orientable} ($\mathsf T_1$) if it admits a coherent
orientation.

The no-fold condition is a property of an oriented product structure. An oriented product
structure $(\Pi,O)$ is \emph{no-fold} ($\mathsf T_2$) if no two distinct horizontal graph edges
incident to a common vertex share both color and direction with respect to $O$; equivalently, no
vertex has two outgoing edges of one color or two incoming edges of one color \cite[(T2)]{GargMineyev}.
We say that $\Pi$ is an \emph{orientable no-fold product structure} when it admits a coherent
orientation $O$ for which $(\Pi,O)$ is no-fold. In every theorem and proof below involving such a
structure, one such orientation is fixed and all signed-color data are formed from it.
\end{definition}

\begin{definition}[middle link and signed colors]
After a coherent orientation $O$ has been fixed, for each color $c$ there are two \emph{signed
colors} $(c,\out)$ and $(c,\inn)$. The \emph{middle link} $L_1$ is the bipartite graph with parts
$A\sqcup B$ and the set of signed
colors, where a directed horizontal edge $x\to y$ of color $c$ contributes the incidences
$\{x,(c,\out)\}$ and $\{y,(c,\inn)\}$ \cite[Definition 11]{GargMineyev}. For $v\in A\sqcup B$
write $\Sigma(v)$ for the set of signed colors incident to $v$ in $L_1$, and for a signed
color $s$ set $A_s=\{a\in A:s\in\Sigma(a)\}$, $B_s=\{b\in B:s\in\Sigma(b)\}$. The graph
$L_1$ is a simple bipartite incidence graph: repeated contributions to the same pair $(v,s)$
produce one incidence. We also write $s\ni v$ as shorthand for $s\in\Sigma(v)$.
\end{definition}

\begin{lemma}[collapsed middle link and endpoint incidences]\label{lem:collapsed-middle}
Fix a coherent orientation $O$. In the Garg--Mineyev quotient links, all cell-induced horizontal
occurrences with the same unordered endpoint pair represent one collapsed horizontal edge. This
collapsed edge has one color and one chosen direction. The middle link used here, and the middle
link entering $\mathsf T_4$, is the simple bipartite incidence graph obtained from these oriented
collapsed colored edges: a directed collapsed edge $x\to y$ of color $c$ contributes the single
incidences $x-(c,\out)$ and $y-(c,\inn)$, and repeated contributions to the same pair $(v,s)$ are
identified. If $(\Pi,O)$ is no-fold, then for every vertex $v$ the map from incident collapsed
horizontal edges at $v$ to signed colors in $\Sigma(v)$ is injective and hence is a bijection onto
$\Sigma(v)$.
\end{lemma}

\begin{proof}
The first assertions are precisely the quotient-link convention recalled above and the
Garg--Mineyev construction of the middle link from directed colored horizontal edges
\cite[Definitions~5--6, \S\S2.4, 2.6, Definition~11]{GargMineyev}. The graph convention is simple:
a repeated contribution to the same incidence pair is one edge of the incidence graph. At a fixed
vertex $v$, each incident collapsed horizontal edge contributes exactly one signed color according
to whether $v$ is its tail or its head. The no-fold condition forbids two distinct incident
collapsed edges of the same color and the same direction at $v$, hence forbids two such edges from
contributing the same signed color. This gives injectivity, and surjectivity onto $\Sigma(v)$ is by
definition of $\Sigma(v)$.
\end{proof}

\begin{lemma}[Garg--Mineyev quotient dictionary]\label{lem:gm-dictionary}
For the product structures considered in this paper, the objects just defined are the
Garg--Mineyev objects used in the taiko conditions: orientability is $\mathsf T_1$, no-fold is
$\mathsf T_2$, the graphs $L_A,L_B$ are the horizontal links, $L_1$ is the middle link, and
$\mathsf T_4$ is the triple-girth alternative stated in the introduction. More explicitly, the
paper's occurrence graph is a way to compute the Garg--Mineyev quotient color map: repeated
cell-induced horizontal occurrences with the same unordered endpoint pair are identified in
$\bar E_A\sqcup\bar E_B$, colors are equivalence classes of these collapsed horizontal edges, a
coherent orientation is a single direction on each collapsed colored edge compatible with every
occurrence in every cell, and $L_1$ is formed from the signed colors of these oriented collapsed
edges. Consequently, by Lemma~\ref{lem:collapsed-middle}, all signed-color incidences and all degrees
in the sequel are incidences and degrees of collapsed simple horizontal edges. In particular, if $(\Pi,O)$ is no-fold, then
$|\Sigma(v)|=\deg(v)$ means degree in the collapsed horizontal link, not multiplicity of
cell-induced occurrences.
\end{lemma}

\begin{proof}
Garg--Mineyev first pass from cell-induced horizontal occurrences to the unordered collapsed
horizontal edge sets $\bar E_A$ and $\bar E_B$; their directed set $E_{AB}$ consists of the two
orientations of each collapsed edge, and their orientations are sections of this directed double
cover satisfying the same cell-compatibility rule. Their color relation is generated by the two
horizontal edges of a cell, after this quotient, and their middle link is formed from directed
colored edges by adding the incidences tail--$(c,\out)$ and head--$(c,\inn)$
\cite[Definitions~5--6, \S\S2.4, 2.6, Definition~11]{GargMineyev}. These are exactly the
operations encoded above by the occurrence graph plus the repeated-endpoint collapse. The no-fold
condition \cite[(T2)]{GargMineyev} is therefore imposed on incident collapsed edges of a common
vertex. Under no-fold, an incident collapsed edge contributes exactly one signed color at its
endpoint and two distinct incident collapsed edges cannot contribute the same signed color there.
This gives the stated degree identification.
\end{proof}

\begin{definition}[patterns and triple-girth]
A \emph{pattern} at a vertex of $A\sqcup B$ is an unordered pair of distinct signed colors incident
to that vertex. The Garg--Mineyev condition $\mathsf T_3$ is the no-repeated-pattern condition: the
same unordered pair of signed colors is not realized at two different vertices. Under the simple
middle-link convention of Lemma~\ref{lem:collapsed-middle}, and in the no-fold setting used below,
a repeated pattern is equivalent to a $4$-cycle in $L_1$. Thus throughout the nonexistence argument
$\mathsf T_3$ is automatic whenever $\girth(L_1)\ge 6$.

For any graph $G$ set $\halfgirth(G)=\girth(G)/2$, with $\infty/2=\infty$. Thus
$\halfgirth(G)\ge q$ is equivalent to $\girth(G)\ge2q$.

The triple-girth condition $\mathsf T_4$ is the disjunction of the three regimes
\[
        (p,q)\in\{(6,3),(4,4),(3,6)\},\qquad
        \girth(L_{AB})\ge p,\quad \halfgirth(L_1)\ge q,
\]
Equivalently, the three branches demand
$(\girth(L_{AB}),\girth(L_1))\ge(6,6)$, $(4,8)$, or $(3,12)$, respectively.
\end{definition}

\paragraph{External assembly input.}
We use Garg--Mineyev's Theorem~16 \cite[Theorem~16]{GargMineyev} (equivalently Theorem~27 of
Mineyev's origami paper \cite{MineyevOrigami}) as a black box. In the notation of the present
paper, the inputs relevant to the $\F_2$ assembly are:
\begin{enumerate}[label=\textup{(\alph*)},leftmargin=2.6em,itemsep=1pt]
\item an even or odd product structure with $m,n\ge2$;
\item orientability $\mathsf T_1$ and no-fold $\mathsf T_2$, interpreted on the collapsed
horizontal links as in Lemma~\ref{lem:gm-dictionary};
\item the triple-girth alternative $\mathsf T_4$, namely the disjunction of the
alternatives $(6,3)$, $(4,4)$, and $(3,6)$ defined above.
\end{enumerate}
This is Garg--Mineyev's shorthand
$\girth(6,3)\allowbreak(4,4)\allowbreak(3,6)$ for the same triple-girth condition. The exact part
of Garg--Mineyev Theorem~16 used here is the equivalence, in their notation, between the
orientation-plus-triple-girth formulation $(2)$ and the orientation/no-fold/triple-girth
formulation $(2')$. Thus every assembly input in that theorem whose hypotheses are stated without no-fold is,
by the theorem itself, equivalent to one satisfying the no-fold formulation obstructed here. The
graphs in these formulations are the collapsed horizontal and middle links identified in
Lemmas~\ref{lem:collapsed-middle} and~\ref{lem:gm-dictionary}. Under
these equivalent hypotheses the external theorem proves nondegeneracy and torsion-freeness of the
associated universal groups; in the even case the assembled elements give a zero-divisor
counterexample over $\F_2$, and in the odd case they give a nontrivial-unit counterexample over
$\F_2$. If the same structure also admits the Garg--Mineyev signature required in that theorem,
the corresponding counterexample is obtained over any ring with unity; this signature requirement
is an additional hypothesis only for the arbitrary-ring upgrade. We use this assembly theorem
only as an external input for Corollary~\ref{cor:route}; the proof below is the combinatorial
obstruction showing that the required orientable no-fold triple-girth input structure does not
exist. The internal definition of signature is not used anywhere in the obstruction.

Thus $A_{(c,\out)}$ is the set of tails of color-$c$ $A$-edges and $A_{(c,\inn)}$ the set of
heads; we write $x_c$ for the number of color-$c$ edges of $L_A$ and $y_c$ for that of $L_B$,
and $e_A=\sum_c x_c$, $e_B=\sum_c y_c$ for the edge counts. All degrees $\deg_{L_A}$ and $\deg_{L_B}$ are total degrees in the underlying undirected collapsed simple horizontal graphs. A color is \emph{active} if it occurs
in a $2$-cell; for an active color both $x_c$ and $y_c$ are positive, and all sums over colors
below are over active colors.

\begin{convention}[fixed orientation]
Throughout a proof, an orientable no-fold product structure is considered together with one fixed
coherent orientation $O$ for which $(\Pi,O)$ is no-fold. The signed colors, the middle link, and the
sets $A_s,B_s$ are formed from this fixed choice. Thus a hypothetical structure satisfying
$\mathsf T_1,\mathsf T_2,\mathsf T_4$ is ruled out by applying the arguments below to the coherent
orientation appearing in that hypothesis; no comparison among different coherent orientations is
needed. Reversing all edges of a color only renames $(c,\out)$ and $(c,\inn)$ and leaves every
cardinality used below unchanged.
\end{convention}

The next lemma records the elementary consequences of the quotient dictionary that will be used
throughout. It is a statement about the collapsed horizontal graphs, not about cell occurrences: a
simple graph edge may be induced by several cells, but after collapse it has a unique color and,
once a coherent orientation is fixed, a unique direction. This oriented-simple compatibility is not
a consequence of the bare board partition alone; it is part of the Garg--Mineyev horizontal-link
quotient recalled in Lemma~\ref{lem:gm-dictionary}.

\begin{lemma}[oriented-simple]\label{lem:simple}
For an orientable product structure, after repeated horizontal-edge occurrences are identified as
in Lemma~\ref{lem:gm-dictionary}, the horizontal graphs $L_A,L_B$ are simple edge-colored graphs
and each simple edge has a single direction; in particular $\girth(L_{AB})\ge 3$, and
$e_A\le\binom m2$, $e_B\le\binom n2$. If moreover $\Pi$ is no-fold, then for every $a\in A$ and
every $b\in B$,
\[
        \abs{\Sigma(a)}=\deg_{L_A}(a)\le m-1,\qquad
        \abs{\Sigma(b)}=\deg_{L_B}(b)\le n-1.
\]
\end{lemma}

\begin{proof}
By the occurrence quotient defined above, $L_A$ and $L_B$ are the underlying simple graphs obtained
by identifying all occurrences with the same unordered pair. The color of a collapsed edge is also
well-defined by construction, because occurrences of the same unordered pair are generators of the
same color component. By Lemma~\ref{lem:gm-dictionary}, the orientation is a function on these collapsed colored
edges: every occurrence of a collapsed edge is read with the same direction. Thus a vertex has at
most $m-1$ horizontal neighbors and $e_A\le\binom m2$. A signed color $(c,\out)\in\Sigma(a)$ records an outgoing
color-$c$ simple edge at $a$, and $(c,\inn)$ an incoming one; by no-fold---the condition
$\mathsf T_2$ imposed on collapsed simple horizontal edges---there is at most one incident simple
edge of each signed color at $a$. Conversely, every incident simple edge contributes
exactly one signed color at $a$. Hence signed colors at $a$ are in bijection with incident simple
edges, so $\abs{\Sigma(a)}=\deg_{L_A}(a)$. The $B$-side identity is identical.
\end{proof}

\begin{lemma}[same signed orientation, distinct tails]
\label{lem:distincttails}
Fix one side, $A$ or $B$, and a color $c$. Two horizontal-edge occurrences of color $c$ with the
same signed orientation---that is, both contribute $(c,\out)$ at their tails, equivalently both are
read tail-to-head in the coherent orientation---and with distinct tails induce distinct collapsed
simple graph edges; the same holds with ``tails'' and $(c,\out)$ replaced by ``heads'' and
$(c,\inn)$.
\end{lemma}

\begin{proof}
If two such occurrences collapsed to the same simple graph edge, then they would have the same
underlying unordered pair. Distinct tails would force the two orientations of that unordered pair
to be opposite, contradicting Lemma~\ref{lem:simple}, which says that the collapsed edge has a
single direction in an orientable product structure. The head version is identical.
\end{proof}

Every branch of $\mathsf T_4$ has $p\ge 3$, hence $\girth(L_{AB})\ge 3$; we use
Lemma~\ref{lem:simple} without further comment.

\section{The middle link as an incidence structure}\label{sec:structure}

Throughout this section $\Pi$ is orientable and no-fold, even or odd in the standing taiko sense,
with $m,n\ge2$, and one coherent no-fold orientation is fixed. The two lemmas preceding
Theorem~\ref{thm:structure} assume $\girth(L_1)\ge 6$; the theorem itself records the equivalence
between that assumption and near-disjointness, together with the structural consequences. By
Lemma~\ref{lem:collapsed-middle}, $L_1$ is the simple bipartite incidence graph of collapsed
signed-color incidences, so the girth hypothesis enters through one observation: a $4$-cycle in
$L_1$ is a configuration $u-s-v-t-u$ with $u\ne v$ in $A\sqcup B$ and $s\ne t$ signed colors,
i.e.\ two distinct points carrying two common signed colors.

\begin{lemma}[no double incidence]\label{lem:nodouble}
If $\girth(L_1)\ge 6$ then no two distinct vertices of $A\sqcup B$ carry two common signed
colors. Equivalently, for distinct $s\ne t$, $\ \abs{A_s\cap A_t}+\abs{B_s\cap B_t}\le 1$.
\end{lemma}

\begin{proof}
Two distinct points on two common signed colors form a $4$-cycle, contradicting
$\girth(L_1)\ge 6$. The quantity $\abs{A_s\cap A_t}+\abs{B_s\cap B_t}$ counts the points
carrying both $s$ and $t$.
\end{proof}

\begin{lemma}[cross-covering]\label{lem:cross}
If $\Pi$ is orientable and no-fold with $\girth(L_1)\ge 6$, then
$\abs{\Sigma(a)\cap\Sigma(b)}=1$ for every non-defect position $(a,b)$. In the odd case
$\abs{\Sigma(a_0)\cap\Sigma(b_0)}=0$ for the defect.
\end{lemma}

\begin{proof}
A non-defect position $(a,b)$ lies in a $2$-cell $\{(a,b),(a',b')\}$ of some color $c$. If its
$A$-edge is oriented $a\to a'$, orientability orients the $B$-edge $b\to b'$, so both $a,b$
carry $(c,\out)$; if $a'\to a$ then both carry $(c,\inn)$. Hence
$\Sigma(a)\cap\Sigma(b)\ne\varnothing$, and by Lemma~\ref{lem:nodouble} it is a singleton.

For the defect, set $\eps=\abs{\Sigma(a_0)\cap\Sigma(b_0)}\in\{0,1\}$
(Lemma~\ref{lem:nodouble}). The sum $\sum_s\abs{A_s}\abs{B_s}=\sum_{(a,b)}\abs{\Sigma(a)\cap
\Sigma(b)}$ counts, over all positions, the common signed colors; the $mn-1$ non-defect
positions contribute $1$ each, and the defect contributes $\eps$, so the sum equals
$mn-1+\eps$. On the other hand, by no-fold the $x_c$ tails and the $x_c$ heads of the color-$c$
$A$-edges are each distinct, so $\abs{A_{(c,\out)}}=\abs{A_{(c,\inn)}}=x_c$, and likewise
$\abs{B_{(c,\out)}}=\abs{B_{(c,\inn)}}=y_c$; grouping the signed colors by color, the same sum equals
$\sum_c\bigl(x_cy_c+x_cy_c\bigr)=2\sum_c x_cy_c$, which is even. Since $mn$ is odd in the odd
case, $mn-1$ is even, so $\eps$ is even, whence $\eps=0$.
\end{proof}

\begin{theorem}[Structure Theorem]\label{thm:structure}
Let $\Pi$ be an orientable no-fold product structure, even or odd in the standing taiko sense, with
$\abs A,\abs B\ge 2$. Then, with $L_1$ understood as the simple bipartite incidence graph of
Lemma~\ref{lem:collapsed-middle}, $\girth(L_1)\ge 6$ if and only if any two distinct signed colors
share at most one point. When this holds, the rectangles
$A_s\times B_s$ partition $(A\times B)\setminus\defect$, and the following conclusions hold:
\begin{enumerate}[label=\textup{(\roman*)},leftmargin=2.6em]
\item For each $a\in A$, $\sum_{s\ni a}|B_s|=n-\one_{a\in D_A}$, and for each $b\in B$,
$\sum_{s\ni b}|A_s|=m-\one_{b\in D_B}$.
\item For distinct signed colors $s\ne t$,
$|A_s\cap A_t|+|B_s\cap B_t|\le1$.
\item The signed colors pair as $(c,\out),(c,\inn)$ with
$|A_{(c,\out)}|=|A_{(c,\inn)}|=:x_c$ and
$|B_{(c,\out)}|=|B_{(c,\inn)}|=:y_c$.
\item The product identity holds:
$\sum_s |A_s||B_s|=mn-\delta$, equivalently $2\sum_c x_cy_c=mn-\delta$.
\item The half-size bounds hold:
$x_c\le\lfloor(m+1)/2\rfloor$ and $y_c\le\lfloor(n+1)/2\rfloor$.
\item The pressure inequalities hold:
$y_c(m-x_c)\le2e_A-x_c+\delta$ and
$x_c(n-y_c)\le2e_B-y_c+\delta$.
\item The Fisher bounds hold:
$\sum_c x_c(x_c-1)\le\binom m2$ and
$\sum_c y_c(y_c-1)\le\binom n2$.
\item If $r$ is the number of signed colors and $d_A(a)=\deg_{L_A}(a)$,
$d_B(b)=\deg_{L_B}(b)$, then
$\sum_{a\in A}\binom{d_A(a)}2+\sum_{b\in B}\binom{d_B(b)}2\le\binom r2$.
\end{enumerate}
\end{theorem}

\begin{proof}
\emph{Equivalence.} By Lemma~\ref{lem:collapsed-middle}, after the fixed no-fold coherent orientation is chosen, $L_1$ is the simple bipartite incidence graph of collapsed signed-color incidences; hence its only possible cycles below length $6$ are $4$-cycles. A $4$-cycle in $L_1$ is a pair of points together with
a pair of signed colors, each point incident to each color; thus $L_1$ has a $4$-cycle if and
only if some two points share two signed colors, equivalently some two distinct signed colors share two
points. Since $L_1$ is bipartite its girth is at least $6$ exactly when it has no $4$-cycle,
which is the stated condition. The forward implication is Lemma~\ref{lem:nodouble}; the rest of
the theorem records the structure present when it holds.

\emph{Partition.} If $(a,b)\in(A_s\times B_s)\cap(A_t\times B_t)$ with $s\ne t$ then
$\abs{A_s\cap A_t}+\abs{B_s\cap B_t}\ge 2$, contradicting Lemma~\ref{lem:nodouble}; so the
rectangles are disjoint. By Lemma~\ref{lem:cross} each non-defect position lies in exactly one
$A_s\times B_s$ and the defect in none. This is the partition, and (ii) is
Lemma~\ref{lem:nodouble}.

\emph{(i) Tiling.} Fix $a$. For each non-defect $(a,b)$ let $s_b$ be the unique signed color
in $\Sigma(a)\cap\Sigma(b)$; then $s_b\ni a$ and $b\in B_{s_b}$. For fixed $s\ni a$ the set
$\{b:s_b=s\}$ is $B_s$: in the odd case with $a=a_0$, Lemma~\ref{lem:cross} gives
$\abs{\Sigma(a_0)\cap\Sigma(b_0)}=0$, so $b_0\notin B_s$ for every $s\ni a_0$. These sets are
disjoint over $s$ (a common $b$ would put $a,b$ on two signed colors), so
$\sum_{s\ni a}\abs{B_s}=n-\one_{a\in D_A}$. The dual is symmetric, using the same defect argument
on the $B$-side.

\emph{(iii) Balance.} By no-fold each color-$c$ $A$-edge has a distinct tail and a distinct
head, so $\abs{A_{(c,\out)}}=\abs{A_{(c,\inn)}}=x_c$; likewise on $B$.

\emph{(iv) Product identity.} Summing over the partition gives
$\sum_s\abs{A_s}\abs{B_s}=mn-\delta$, and grouping by color with (iii) gives $2\sum_c x_cy_c$.

\emph{(v) Half-size.} Let $T_c,H_c\subseteq A$ be the tails and heads of color-$c$ edges,
$\abs{T_c}=\abs{H_c}=x_c$. If two distinct vertices lay in $T_c\cap H_c$ they would carry both
$(c,\out),(c,\inn)$, contradicting Lemma~\ref{lem:nodouble}; so $\abs{T_c\cap H_c}\le 1$ and
$m\ge\abs{T_c\cup H_c}=2x_c-\abs{T_c\cap H_c}\ge 2x_c-1$. The bound on $y_c$ is dual.

\emph{(vi) Pressure.} Fix a signed color $s$ over $c$, so $\abs{A_s}=x_c$, $\abs{B_s}=y_c$.
For $b\in B_s$, the dual tiling gives $\sum_{t\in\Sigma(b)}\abs{A_t}=m-\one_{b\in D_B}$;
removing the term $t=s$ (note $s\in\Sigma(b)$),
\[
        \sum_{b\in B_s}\sum_{t\in\Sigma(b)\setminus\{s\}}\abs{A_t}
        =y_c(m-x_c)-\abs{B_s\cap D_B}.
\]
Exchanging summation and using that each $t\ne s$ is carried together with $s$ by at most one
$b\in B_s$ (Lemma~\ref{lem:nodouble}), the left side is
$\sum_{t\ne s}\abs{A_t}\,\abs{B_s\cap B_t}\le\sum_{t\ne s}\abs{A_t}$. Here
$\sum_t\abs{A_t}=\sum_a\abs{\Sigma(a)}=2e_A$, since each $A$-edge contributes exactly one tail
incidence $(c,\out)$ and one head incidence $(c,\inn)$; hence $\sum_{t\ne s}\abs{A_t}=2e_A-x_c$.
Since $\abs{B_s\cap D_B}\le\abs{D_B}=\delta$, this gives $y_c(m-x_c)\le 2e_A-x_c+\delta$, and the
second inequality is dual.

\emph{(vii) Fisher bounds.} These specialize the classical Fisher-type incidence inequality
\cite{vanLintWilson}. By Lemma~\ref{lem:nodouble} any two $A$-points lie on at most one
common signed color, so the pairs of $A$-points counted with multiplicity by the signed colors
are distinct: $\sum_s\binom{\abs{A_s}}2\le\binom m2$. Grouping the two signed colors over each
color $c$ and using $\abs{A_{(c,\out)}}=\abs{A_{(c,\inn)}}=x_c$ turns the left side into
$2\sum_c\binom{x_c}2=\sum_c x_c(x_c-1)$, giving the first bound; the $B$-side is identical.

\emph{(viii) Dual Fisher.} Near-disjointness (ii) says any two distinct signed colors share at most one
point of $A\sqcup B$, so the number of pairs of distinct signed colors that meet is at most
$\binom r2$, where $r$ is the number of signed colors. Counting these pairs instead by their common
point, a vertex $v$ lying on
$d_v=\abs{\Sigma(v)}$ signed colors accounts for $\binom{d_v}2$ of them, and no pair is counted
twice; hence $\sum_{v\in A\sqcup B}\binom{d_v}2\le\binom r2$. By the oriented-simple property
$\abs{\Sigma(a)}=\deg_{L_A}(a)$ and $\abs{\Sigma(b)}=\deg_{L_B}(b)$, which is the stated bound.
\end{proof}

\begin{figure}[t]
\centering
\begin{tikzpicture}[scale=0.6]
  \fill[blue!16]  (0,3) rectangle (2,6);
  \fill[red!13]   (3,1) rectangle (6,4);
  \draw[step=1,gray!55,very thin] (0,0) grid (6,6);
  \draw[blue!70,thick] (0,3) rectangle (2,6);
  \draw[red!75,thick]  (3,1) rectangle (6,4);
  \foreach \c in {1,...,6} {\node[gray,font=\scriptsize] at (\c-0.5,6.32) {$b_{\c}$};}
  \foreach \r in {1,...,6} {\node[gray,font=\scriptsize] at (-0.38,6.5-\r) {$a_{\r}$};}
  \node[blue!70,font=\small] at (1,4.5)   {$A_s\times B_s$};
  \node[red!75,font=\small]  at (4.5,2.5) {$A_{s'}\!\times\!B_{s'}$};
  \draw[->,gray] (8.05,3.5) -- (6.05,3.5);
  \node[gray,font=\scriptsize,anchor=west] at (8.1,3.5) {shared row $a_3$};
\end{tikzpicture}
\caption{The Structure Theorem reads $\girth(L_1)\ge 6$ as a partition of the board $A\times B$
(rows indexed by $A$, columns by $B$) into combinatorial rectangles $A_s\times B_s$, one per
signed color $s$. Near-disjointness forces two rectangles to share at most one row in total or at
most one column in total, and never both; here $A_s\times B_s$ and $A_{s'}\times B_{s'}$ share only
the row $a_3$, hence no common position. When $mn$ is odd a single position, the defect, is left
uncovered.}
\label{fig:rectangles}
\end{figure}
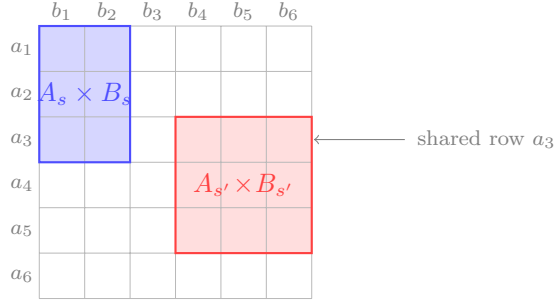

We record the Moore bound for later use, in the irregular form of Alon, Hoory, and
Linial \cite{AlonHooryLinial}, with its sharpness.

\begin{lemma}[Moore bound, girth $6$]\label{lem:moore}
Let $G$ be a finite simple graph with $r$ vertices, minimum degree at least $2$, average
degree $d$, and girth at least $6$. Then $r\ge 2\bigl(1+(d-1)+(d-1)^2\bigr)$, equivalently
\[
        e(G)\le E_r:=\Bigl\lfloor\tfrac r4\bigl(1+\sqrt{2r-3}\bigr)\Bigr\rfloor .
\]
Equality holds for the incidence graph of a projective plane of order $q$, where $d=q+1$ and
$r=2(q^2+q+1)$.
\end{lemma}

\begin{proof}
The first inequality is \cite{AlonHooryLinial}. Solving $d^2-d+1\le r/2$ gives
$d\le\tfrac12(1+\sqrt{2r-3})$, so $e(G)=rd/2\le\tfrac r4(1+\sqrt{2r-3})$, and $e(G)$ is an
integer. For a plane of order $q$ the Levi graph is $(q+1)$-regular bipartite of girth $6$ on
$2(q^2+q+1)$ vertices, where $(q+1)^2-(q+1)+1=q^2+q+1=r/2$.
\end{proof}

For any positive integer $r$ set $h_r=\lfloor\tfrac{r+1}2\rfloor$ and
$E_r=\lfloor\tfrac r4(1+\sqrt{2r-3})\rfloor$ throughout.

\begin{remark}
By (i)--(iii) the middle link is the incidence graph of a partial linear space: a cross-free
family of rectangles tiling the board. The pressure inequalities (vi) are strictly stronger
than the Cauchy--Schwarz inequalities available from (i)--(iv) alone; they encode that the two
signed colors over a color are balanced and that lines meet in at most one point. This extra
strength is what the diagonal analysis of Section~\ref{sec:B} requires
(Remark~\ref{rem:necessary}).
\end{remark}

\section{Small supports}\label{sec:small}

In this section $\Pi$ is orientable and no-fold, even or odd in the standing taiko sense, with $\girth(L_1)\ge 6$, so the Structure
Theorem applies and $L_A,L_B$ are oriented simple. We dispose of supports with a class of size
at most three, and force $\abs A,\abs B\ge 6$ in the $(6,3)$ branch.

\begin{lemma}[no thin class]\label{lem:m2}
There is no orientable no-fold product structure, even or odd in the standing taiko sense, with
$m,n\ge2$, $\girth(L_1)\ge6$, and $\min\{m,n\}\le 2$.
\end{lemma}

\begin{proof}
By symmetry suppose $m=\min\{m,n\}=2$ (we have $m,n\ge 2$). The odd case is impossible by
parity: $mn=2n$ is even while $\delta=1$ needs $mn$ odd. So $\delta=0$, and $P$ partitions
$A\times B$ into $2$-cells. Each cell meeting row $a_1$ has its other position in row $a_2$,
and its $A$-edge is the unique edge of the oriented-simple graph $L_A$ on two vertices, say
$a_1\to a_2$ of color $c$. By orientability every cell's $B$-edge occurrence has color $c$, oriented from
its $a_1$-row endpoint to its $a_2$-row endpoint, so the cells determine a fixed-point-free
permutation $\pi$ of $B$: it is defined from the row-$a_1$ positions, and it is bijective because
the row-$a_2$ positions are also partitioned exactly once. The occurrences are $b\to\pi(b)$.
Lemma~\ref{lem:distincttails} makes the occurrences with distinct tails distinct collapsed simple
edges, so each vertex is the tail of a color-$c$ simple edge. Applying the same statement to heads
shows that each vertex is also the head of a color-$c$ simple edge. If $n=2$, the permutation $\pi$
is a transposition, so the two occurrences would force the unique collapsed $B$-edge to carry two
opposite coherent orientations, contradicting the oriented-simple convention. Hence $n\ge3$.
If $\pi$ has a $2$-cycle $(b\ b')$, the two occurrences $b\to b'$ and $b'\to b$ force opposite
orientations on the same collapsed $B$-edge, again contradicting Lemma~\ref{lem:simple}. Hence all
cycles of $\pi$ have length at least $3$. Therefore the outgoing edge $b\to\pi(b)$ and the incoming
edge $\pi^{-1}(b)\to b$ are distinct collapsed simple edges for every $b$, so every $b$ carries
both signed colors:
$\{(c,\out),(c,\inn)\}\subseteq\Sigma(b)$.
Two distinct vertices $b,b'$ then carry the common signed colors $(c,\out),(c,\inn)$, a
$4$-cycle in $L_1$, contradicting $\girth(L_1)\ge 6$.
\end{proof}

\begin{lemma}[type $(3,3)$]\label{lem:threethree}
Let $\Pi$ be an orientable no-fold product structure, even or odd in the standing taiko sense, of type $(3,3)$ with $\girth(L_1)\ge 6$.
If in addition $\girth(L_{AB})\ge 4$ \emph{or} $\girth(L_1)\ge 8$, then $\Pi$ does not exist.
In particular no $(3,3)$ structure satisfies any branch of $\mathsf T_4$, since each branch
has $\girth(L_{AB})\ge 4$ or $\girth(L_1)\ge 8$.
\end{lemma}

\begin{proof}
Here $mn=9$ is odd, so $\delta=1$ and the Structure Theorem gives the product identity
$\sum_c x_cy_c=(9-1)/2=4$. Half-size gives $x_c,y_c\le h_3=2$, and $L_A,L_B$ simple on $3$
vertices gives $e_A=\sum_c x_c\le 3$ and $e_B=\sum_c y_c\le 3$. A color contributes
$x_cy_c\in\{1,2,4\}$. If some color contributes $4$, i.e.\ $(x_c,y_c)=(2,2)$, it is the only
color (any other would make the sum exceed $4$). Otherwise every contribution is $\le 2$; to
sum to $4$ under $\sum x_c\le 3,\ \sum y_c\le 3$ the only possibility is two colors with
$(x_c,y_c)=(2,1)$ and $(1,2)$. Thus there are exactly two admissible color profiles.

\emph{Profile $\{(2,2)\}$.} A single color with $x_c=2$ on three vertices forces $L_A$ to be
the path $a_1\!-\!a_2\!-\!a_3$ (the only simple graph with two edges on three vertices), both
edges of color $c$. By no-fold the middle vertex $a_2$ cannot be the tail of both edges or the
head of both, so it is the tail of one and the head of the other; hence $a_2$ carries both
$(c,\out)$ and $(c,\inn)$. The same holds
for the middle vertex of the $B$-path. These two vertices carry the two common signed colors,
a $4$-cycle in $L_1$, contradicting $\girth(L_1)\ge 6$.

\emph{Profile $\{(2,1),(1,2)\}$.} Here $e_A=e_B=3$, so $L_A,L_B$ are triangles. If
$\girth(L_{AB})\ge 4$ then triangles are forbidden, a contradiction, and we are done. Assume
instead $\girth(L_1)\ge 8$. In this profile the four signed colors $(c,\out),(c,\inn)$
($c\in\{c_1,c_2\}$) each have degree $\abs{A_s}+\abs{B_s}=x_c+y_c=3$ in $L_1$, while each of
the six points of $A\sqcup B$, being a vertex of a triangle, has degree
$\abs{\Sigma(\cdot)}=2$. So $L_1$ is a bipartite graph on six degree-$2$ points and four
degree-$3$ lines. Because $\girth(L_1)\ge 6$ no two points share two signed colors, so the map
sending a point to its $2$-element set of incident lines is injective; there are exactly
$\binom 42=6$ such pairs and six points, so every pair occurs, and $L_1$ is the
\emph{subdivision of the complete graph $K_4$} on the four lines (the six points are its
subdivided edges). But $K_4$ contains a triangle, whose subdivision is a $6$-cycle in $L_1$;
hence $\girth(L_1)=6<8$, a contradiction.

Finally, each branch of $\mathsf T_4$ meets the hypothesis: $(6,3)$ and $(4,4)$ have
$\girth(L_{AB})\ge 4$, and $(3,6)$ has $\girth(L_1)\ge 12\ge 8$.
\end{proof}

\begin{lemma}[degree-one obstruction]\label{lem:degreeone}
Let $\Pi$ be an orientable no-fold product structure, even or odd in the standing taiko sense, with $\girth(L_1)\ge6$ and $m,n\ge3$. If a
vertex of $L_A$ has degree one, then $n\le3$; symmetrically, if a vertex of $L_B$ has degree one,
then $m\le3$.
\end{lemma}

\begin{proof}
Suppose $a\in A$ has $\deg_{L_A}(a)=1$, with unique edge $\{a,a'\}$ of color $c$. If this edge is
oriented $a\to a'$, then every non-defect position $(a,b)$ lies in a $2$-cell whose $A$-edge is
this one, so by orientability $b$ is the tail of a color-$c$ $B$-edge occurrence. As $b$ varies over
the columns for which $(a,b)$ is non-defect, these tails are distinct, and
Lemma~\ref{lem:distincttails} makes the resulting color-$c$ $B$-edges distinct collapsed simple
edges. If the unique edge is oriented $a'\to a$, the same argument uses heads and incoming edges,
together with the head version of Lemma~\ref{lem:distincttails}. In either case there are
$n-\one_{a\in D_A}$ such columns, hence $y_c\ge n-\one_{a\in D_A}\ge n-1$. The half-size bound
gives $y_c\le h_n=\lfloor(n+1)/2\rfloor$, so $n-1\le\lfloor(n+1)/2\rfloor$ and $n\le3$. The
$B$-side statement is the same argument with $A$ and $B$ interchanged.
\end{proof}

\begin{proposition}[the $(6,3)$ branch has $\abs A,\abs B\ge 6$]\label{prop:six}
In the $(6,3)$ branch, any orientable no-fold product structure, even or odd in the standing taiko sense, with $m,n\ge 2$, has minimum degree $\ge 2$ in both $L_A$ and
$L_B$, and
$m,n\ge 6$. \textup{(}If no example exists the statement is vacuous.\textup{)}
\end{proposition}

\begin{proof}
By Lemma~\ref{lem:m2} we may assume $m,n\ge 3$. Lemma~\ref{lem:degreeone} shows that a
degree-one vertex of $L_A$ forces $n\le3$, and a degree-one vertex of $L_B$ forces $m\le3$. Since
$m,n\ge 3$, every row and every column contains a non-defect position (at most one position is a
defect), and that position lies in a $2$-cell, so no vertex of $L_A$ or $L_B$ is isolated.

Suppose $\min\{m,n\}=3$, say $m=3$. Then $L_A$ is a graph on three vertices with
$\girth(L_A)\ge 6$, hence acyclic; having no isolated vertex, it has a vertex of degree exactly
$1$. Lemma~\ref{lem:degreeone} gives $n\le 3$, so $n=3$ and $\Pi$ has type $(3,3)$, excluded
by Lemma~\ref{lem:threethree}.

Suppose instead $\min\{m,n\}\ge 4$. Then no vertex has degree $1$ (else $n\le 3$ or $m\le 3$),
and none is isolated, so $L_A,L_B$ have minimum degree at least $2$ and therefore contain
cycles; as $\girth(L_A),\girth(L_B)\ge 6$, each such cycle has at least six vertices, so
$m,n\ge 6$.
\end{proof}

\section{The high middle-link-girth branches: Theorem A}\label{sec:A}

Here $\girth(L_1)\ge 8$, so $L_1$ has no $4$-cycle and no $6$-cycle; all product structures in this section remain even or odd in the standing taiko sense. The key point is that the
rectangles degenerate.

\begin{lemma}[thin lines]\label{lem:thin}
Let $\Pi$ be orientable and no-fold, even or odd in the standing taiko sense, with
$\girth(L_1)\ge 8$, and let $2\le m\le n$ with $n\ge 4$. Then $\abs{A_s}\le 1$ for every
signed color $s$. If moreover $m\ge 3$, then also
$\abs{B_s}\le 1$ for every signed color $s$.
\end{lemma}

\begin{proof}
Since $\girth(L_1)\ge8$ implies $\girth(L_1)\ge6$, the Structure Theorem applies throughout this
proof, in particular the half-size bounds and cross-covering lemma are available. Suppose
$\abs{A_s}\ge 2$, with $a_1,a_2\in A_s$ distinct. For any $b\notin B_s$ with
$(a_1,b),(a_2,b)$ non-defect, Lemma~\ref{lem:cross} gives unique $t_i\in\Sigma(a_i)\cap
\Sigma(b)$; both differ from $s$ (as $b\notin B_s$), and $t_1\ne t_2$ (else $a_1,a_2$ share
$s,t_1$, a $4$-cycle). Then $a_1-s-a_2-t_2-b-t_1-a_1$ is a $6$-cycle in $L_1$, impossible.
Hence every $b$ with $(a_1,b),(a_2,b)$ non-defect lies in $B_s$; these are all $b\notin D_B$,
so $\abs{B_s}\ge n-\abs{D_B}\ge n-1$. But $\abs{B_s}=y_c\le h_n$ and $n-1>h_n$ for $n\ge 4$,
a contradiction. So $\abs{A_s}\le 1$.

Now assume $m\ge 3$ and suppose $\abs{B_s}\ge 2$, with $b_1,b_2\in B_s$ distinct. For any
$a\notin D_A$ with $a\notin A_s$, both $(a,b_1),(a,b_2)$ are non-defect, and the symmetric
argument produces a $6$-cycle $b_1-s-b_2-t_2-a-t_1-b_1$, impossible. Hence every $a\notin D_A$
lies in $A_s$, so $\abs{A_s}\ge m-\abs{D_A}\ge m-1\ge 2$, contradicting $\abs{A_s}\le 1$. So
$\abs{B_s}\le 1$.
\end{proof}

\begin{lemma}[collapse]\label{lem:collapse}
There is no orientable no-fold product structure, even or odd in the standing taiko sense, with $m,n\ge 2$ and $\girth(L_1)\ge 6$ in
which $\abs{A_s}\le 1$ and $\abs{B_s}\le 1$ for every signed color $s$.
\end{lemma}

\begin{proof}
Take $m\le n$ without loss of generality. Each rectangle contains at most one board position, so by the
product identity
\[
\begin{aligned}
        mn-\delta=\sum_s\abs{A_s}\abs{B_s}
        &=\#\{s:\abs{A_s}=\abs{B_s}=1\}\\
        &\le\min\bigl(\,\#\{\abs{A_s}=1\},\ \#\{\abs{B_s}=1\}\,\bigr).
\end{aligned}
\]
Since $\abs{A_s}\le 1$, $\#\{\abs{A_s}=1\}=\sum_s\abs{A_s}=\sum_a\abs{\Sigma(a)}=2e_A$, and
likewise $\#\{\abs{B_s}=1\}=2e_B$. With $e_A\le\binom m2$ (Lemma~\ref{lem:simple}) and
$m\le n$,
\[
        mn-\delta\le\min(2e_A,2e_B)\le 2e_A\le 2\binom m2=m^2-m.
\]
But $n\ge m$ gives $mn-\delta\ge m^2-\delta>m^2-m$ since $m>\delta$, a contradiction.
\end{proof}

\begin{theorem}\label{thm:A}
No orientable no-fold product structure, even or odd in the standing taiko sense, with
$m,n\ge 2$ and $\girth(L_1)\ge 8$ exists. In particular the $(4,4)$ and $(3,6)$ branches of
$\mathsf T_4$ are empty, and every such orientable no-fold product structure has
$\girth(L_1)\le 6$.
\end{theorem}

With the convention that an acyclic graph has girth $\infty$, the last assertion also says that
the middle link of such an orientable no-fold product structure is never a forest.

\begin{proof}
Take $m\le n$. If $m=2$, Lemma~\ref{lem:m2} applies. If $m\ge 3$ and $(m,n)\ne(3,3)$, then
$n\ge 4$ (since $n\ge m\ge 3$, and $n=3$ would force $m=3$). By Lemma~\ref{lem:thin} every
signed color satisfies $\abs{A_s}\le 1$ and $\abs{B_s}\le 1$, and Lemma~\ref{lem:collapse}
gives a contradiction. The remaining shape $(3,3)$ is excluded by Lemma~\ref{lem:threethree}
via its $\girth(L_1)\ge 8$ alternative. This proves the first assertion. The branches $(4,4)$
and $(3,6)$ have $\girth(L_1)\ge 8$ and are therefore empty; and since a structure with
$\girth(L_1)\ge 8$ cannot exist, every structure that does exist has $\girth(L_1)\le 6$
\textup{(}bipartiteness makes the girth even, and by convention a forest has girth $\infty\ge 8$\textup{)}.
\end{proof}

\begin{remark}
The collapse is the degenerate incidence structure in which every line carries at most one
point of each class; the product identity then counts the covered cross-pairs as the
doubly-occupied lines, of which there are too few because the horizontal graphs --- being
simple --- have at most $\binom m2$ and $\binom n2$ edges. No Moore bound is needed in this
regime. The reduction of the possible supports to a single shape $(3,3)$, and the elimination of
$(3,3)$ itself, both use only $\girth(L_1)\ge 8$ and the half-size bound; no horizontal-girth
hypothesis enters at all. Theorem~\ref{thm:A} is thus a statement about the middle link in
isolation: \emph{$\girth(L_1)\le 6$ for every such orientable no-fold product structure}, whatever
its horizontal graphs look like.
\end{remark}

\section{The \texorpdfstring{$(6,3)$}{(6,3)} branch: Theorem B}\label{sec:B}

Here $\girth(L_A),\girth(L_B),\girth(L_1)\ge 6$. By Proposition~\ref{prop:six} we may assume
$m,n\ge 6$; take $m\le n$. We use the product identity (iv), half-size (v), pressure (vi), and
the Moore bounds $e_A\le E_m$, $e_B\le E_n$ (Lemma~\ref{lem:moore}). The finite numerical
comparisons used below are displayed explicitly; all are elementary integer or polynomial
arithmetic, with no appeal to a computer search.

\subsection{The pressure--Fisher bound}

The proof rests on a single second-moment inequality: pressure bounds each $y_c$ by a convex
function of $x_c$, and the Fisher bound controls the second moment of the $x_c$. This is
sharper than bounding $y_c$ by its maximum, and it makes the obstruction non-marginal.

\begin{lemma}[pressure--Fisher bound]\label{lem:pf}
In the $(6,3)$ branch, with $m\le n$, let $E_m$ be the girth-six Moore bound of
Lemma~\ref{lem:moore}, $h_m=\lfloor\tfrac{m+1}2\rfloor$, and
\[
        R(x)=\frac{2E_m+\delta-x}{m-x}\qquad(1\le x\le h_m).
\]
Then, with $\mu=\dfrac{R(h_m)-R(1)}{h_m-1}$,
\[
        \frac{mn-\delta}{2}=\sum_c x_cy_c\ \le\ Q_{m,\delta}:=R(1)\,E_m+\mu\binom m2 ,
\]
where the dependence on $\delta\in\{0,1\}$ enters through $R$ \textup{(}and hence $R(1)$ and
$\mu$\textup{)}. We abbreviate $Q_{m,\delta}$ to $Q_m$ when $\delta$ is fixed by the context.
\end{lemma}

\begin{proof}
Summing over active colors ($x_c\ge 1$), pressure~(vi) and $e_A\le E_m$ give
$y_c(m-x_c)\le 2e_A-x_c+\delta\le 2E_m+\delta-x_c$, that is $y_c\le R(x_c)$. Since
\[
        R'(x)=\frac{2E_m+\delta-m}{(m-x)^2}>0,
        \qquad
        R''(x)=\frac{2(2E_m+\delta-m)}{(m-x)^3}>0
\]
(as $2E_m\ge 2m>m$, since $L_A$ has minimum degree at least $2$ by Proposition~\ref{prop:six},
whence $e_A\ge m$ and $E_m\ge e_A\ge m$), the function $R$ is increasing and convex on
$[1,h_m]$, hence lies below its chord through $x=1$ and $x=h_m$: $R(x)\le R(1)+\mu(x-1)$.
Therefore $x_cy_c\le R(1)x_c+\mu\,x_c(x_c-1)$, and summing with $\sum_c x_c=e_A\le E_m$ and the
Fisher bound~(vii) $\sum_c x_c(x_c-1)\le\binom m2$ gives the claim.
\end{proof}

Because $n\ge m$ gives $\tfrac{mn-\delta}2\ge\tfrac{m^2-\delta}2$, Lemma~\ref{lem:pf} produces a
contradiction whenever $Q_m<\tfrac{m^2-\delta}2$. The next lemma shows this holds for all but
two shapes.

\begin{lemma}\label{lem:pfrange}
$Q_{m,\delta}<\dfrac{m^2-\delta}{2}$ for every even $m\ge 8$ \textup{(}where necessarily
$\delta=0$\textup{)} and every odd $m\ge 9$ \textup{(}for both $\delta=0$ and $\delta=1$\textup{)}.
\end{lemma}

\begin{proof}
\emph{Bridge $8\le m\le 13$.} The integer Moore values are $E_m=9,10,12,14,16,18$ for
$m=8,\dots,13$. Direct evaluation gives, with $\delta=0$ for even $m$ and the binding value
$\delta=1$ for odd $m$,
\[
\begin{array}{c|cccccc}
 m & 8 & 9 & 10 & 11 & 12 & 13\\ \hline
 Q_m & \tfrac{223}{7} & \tfrac{77}{2} & \tfrac{134}{3} & 59 & \tfrac{716}{11} & 80\\[3pt]
 \tfrac{m^2-\delta}{2} & 32 & 40 & 50 & 60 & 72 & 84
\end{array}
\]
so $Q_m<\tfrac{m^2-\delta}2$ in every column.

\emph{Tail $m\ge 14$.} Here $E_m\le E_0:=\tfrac m4(1+s)$ with $s=\sqrt{2m-3}\ge 5$. Since $Q_m$
is increasing in $E_m$ (both $R(1)$ and $\mu$ are; indeed
$\partial R_E(x)/\partial E=2/(m-x)$ and $1/(m-x)$ is increasing in $x$, so the secant slope
increases with $E$; moreover $R_E(1)>0$, so $E\mapsto E R_E(1)$ is nondecreasing), it suffices to bound $Q_m$ with $E_m$ replaced by $E_0$. Writing $m=(s^2+3)/2$, the half-size is $h_m=m/2$ when $m$ is even and
$h_m=(m+1)/2$ when $m$ is odd. For even $m$ (hence $\delta=0$), a direct simplification gives
\[
        \frac{m^2}{2}-Q_m\big|_{E_0}
        =\frac{s^6-6s^5+11s^4-24s^3+35s^2-18s+33}{16(s^2+1)}.
\]
For odd $m$ and $\delta=1$, the corresponding simplification is
\[
        \frac{m^2-1}{2}-Q_m\big|_{E_0}
        =\frac{s^6-6s^5+11s^4-36s^3+15s^2-54s-11}{16(s^2+1)}.
\]
For the even $\delta=0$ case and $s\ge 5$, the substitution $s=t+5$ ($t\ge 0$) turns the numerator into
\[
        t^6+24t^5+236t^4+1196t^3+3200t^2+4032t+1568,
\]
a polynomial with positive coefficients, hence positive; so $Q_m<m^2/2$ for all even
$m\ge 14$. For the odd $\delta=1$ case, the numerator has derivative $6s^5-30s^4+44s^3-108s^2+30s-54$,
which under $s=t+5$ becomes $6t^5+120t^4+944t^3+3552t^2+6000t+2896$ and so is positive for
$s\ge 5$; the numerator is therefore increasing on $[5,\infty)$, and at $s=\sqrt{27}$ (that is,
$m=15$) it equals $28096-16200\sqrt3>0$, since $28096^2>3\cdot 16200^2$. Hence it is positive
for all $s\ge\sqrt{27}$, i.e.\ all odd $m\ge 15$.

Finally, for every odd $m$---in both the bridge and the tail---the case $\delta=0$ follows from
the case $\delta=1$ established above: $Q_{m,\delta}$ is increasing in $\delta$ (both $R(1)$ and
$\mu$ are), so $Q_{m,0}\le Q_{m,1}<\tfrac{m^2-1}2<\tfrac{m^2}2$.
\end{proof}

\begin{remark}[arithmetic reproducibility]
The bridge values in Lemma~\ref{lem:pfrange} are exact rational evaluations of the displayed
pressure--Fisher expression using the integer Moore bounds $E_m$; no numerical rounding enters.
The tail comparisons are polynomial positivity checks after the substitutions shown in the proof.
\end{remark}

\subsection{The two surviving shapes}

\begin{lemma}\label{lem:pfsurv}
There is no $(6,3)$-branch example with $m=6$ or $m=7$.
\end{lemma}

\begin{proof}
\emph{$m=6$ (so $\delta=0$).} Lemma~\ref{lem:pf} gives $3n=\tfrac{6n}{2}=\sum_c x_cy_c\le
Q_6=\tfrac{96}{5}$, so $n\le 6$ and hence $n=6$. Then $\girth(L_A),\girth(L_B)\ge 6$ with minimum degree
$2$ on six vertices forces $L_A=L_B=C_6$: indeed minimum degree at least $2$ gives
$e_A,e_B\ge6$, while Lemma~\ref{lem:moore} gives $e_A,e_B\le E_6=6$, so each graph is
$2$-regular on six vertices and, with girth at least $6$, is a $6$-cycle. Thus $e_A=e_B=6$ and $\sum_c x_cy_c=18$ with
$\sum_c x_c=\sum_c y_c=6$ and $x_c,y_c\le 3$. Equality in
$18=\sum_c x_cy_c\le 3\sum_c x_c=18$ first forces $y_c=3$ for every active color. Since
$\sum_c y_c=6$, there are exactly two active colors; then $\sum_c x_c=6$ and $x_c\le3$ force
$x_c=3$ for both of them. Thus $x_c=y_c=3$ for every active color, so there are exactly two colors
and four signed colors. Every vertex of $A\sqcup B$ then has degree $2$ in $L_1$, so
$\sum_v\binom{d_v}{2}=12$, which contradicts the dual Fisher bound~(viii)
$\sum_v\binom{d_v}{2}\le\binom 42=6$.

\emph{$m=7$.} Since $\delta\equiv mn\pmod 2$, both $\delta=0$ (with $n$ even) and $\delta=1$
(with $n$ odd) are a priori possible, and we treat both. If $\delta=0$, then Lemma~\ref{lem:pf}
gives $\tfrac{7n}{2}\le Q_7=\tfrac{70}{3}$, so $n\le\tfrac{20}{3}<7$, contradicting $n\ge m=7$.
If $\delta=1$, then $n$ is odd and Lemma~\ref{lem:pf} gives $\tfrac{7n-1}{2}\le Q_7=\tfrac{77}{3}$,
so $n\le 7$ and hence $n=7$. Pressure then gives $y_c\le R(h_7)=\tfrac{2E_7-4+1}{7-4}=\tfrac{11}{3}$,
so $y_c\le 3$ for every color, and $\sum_c x_cy_c\le 3\sum_c x_c\le 3E_7=21$, contradicting
$\sum_c x_cy_c=\tfrac{49-1}{2}=24$.
\end{proof}

\begin{theorem}\label{thm:B}
No orientable no-fold product structure, even or odd in the standing taiko sense, with
$m,n\ge 2$ and $\girth(L_A),\girth(L_B),\girth(L_1)\ge 6$ exists. Hence the $(6,3)$ branch of
$\mathsf T_4$ is empty.
\end{theorem}

\begin{proof}
By Proposition~\ref{prop:six}, any such structure has $m,n\ge 6$. Interchange $A$ and $B$ if
necessary and assume $m\le n$. For $m\ge 8$, Lemmas~\ref{lem:pf} and~\ref{lem:pfrange} give
$\tfrac{m^2-\delta}2\le\tfrac{mn-\delta}2\le Q_m<\tfrac{m^2-\delta}2$, a contradiction; the
remaining shapes $m\in\{6,7\}$ are excluded by Lemma~\ref{lem:pfsurv}.
\end{proof}

\begin{remark}
The engine of Theorem~\ref{thm:B} is the product identity weighed against the Fisher
inequalities. Bounding $y_c$ by its maximum $R(h_m)$ --- the one-sided ``strip'' estimate ---
uses only the first moment $\sum_c x_c\le E_m$ and is marginal on the diagonal; the convex
chord in Lemma~\ref{lem:pf} brings in the second moment through Fisher and is what makes the
obstruction strict, reducing the entire branch to two small shapes settled by the dual Fisher
inequality and pressure.
\end{remark}

\section{Synthesis, sharpness, and necessity}\label{sec:synthesis}

The two extremal theorems combine into a single structural statement, of which the
nonexistence of $\mathsf T_4$ taikos is a corollary.

\begin{corollary}[middle-link dichotomy]\label{cor:dichotomy}
Let $\Pi$ be an orientable no-fold product structure, even or odd in the standing taiko sense,
with $m,n\ge 2$. With the convention that forests have girth $\infty$,
$\girth(L_1)\in\{4,6\}$. Moreover, if $\girth(L_1)=6$ then $\girth(L_{AB})\le 5$; equivalently
$\girth(L_{AB})\ge 6$ forces $\girth(L_1)=4$.
\end{corollary}

\begin{proof}
By the standing convention a forest has girth $\infty$, and $L_1$ is bipartite, so its girth is even or infinite. By Theorem~\ref{thm:A} it is
not $\ge 8$ (this also rules out a forest, whose girth is $\infty$), so $\girth(L_1)\in\{4,6\}$.
If $\girth(L_1)=6$ then $\girth(L_1)\ge 6$, so by the contrapositive of Theorem~\ref{thm:B} one
cannot also have $\girth(L_A),\girth(L_B)\ge 6$; thus $\girth(L_{AB})=\min\{\girth(L_A),
\girth(L_B)\}\le 5$. The last clause is the contrapositive.
\end{proof}

The dichotomy constrains \emph{every} orientable no-fold product structure, not only those a
construction would want; the nonexistence of $\mathsf T_4$ taikos is the special case in which
one demands more of the middle link than it can carry.

\begin{proof}[Proof of Theorem~\ref{thm:main}]
Each regime of $\mathsf T_4$ contradicts Corollary~\ref{cor:dichotomy}. The $(6,3)$ branch
demands $\girth(L_1)\ge 6$ and $\girth(L_{AB})\ge 6$; by the dichotomy the first forces
$\girth(L_1)=6$, hence $\girth(L_{AB})\le 5$, contradicting the second. The $(4,4)$ and $(3,6)$
branches demand $\girth(L_1)\ge 8$, contradicting $\girth(L_1)\le 6$ directly. In every case
$\mathsf T_4$ fails.
\end{proof}

\begin{proof}[Proof of Corollary~\ref{cor:route}]
We use the form of the taiko assembly in \cite[Theorem 16]{GargMineyev} (Theorem 27 of
\cite{MineyevOrigami}) for even/odd product structures satisfying the listed hypotheses, with the
horizontal links, middle link, orientability, no-fold condition, and triple-girth condition
identified in Section~\ref{sec:taiko}. That external theorem proves the relevant nondegeneracy
under these hypotheses. In that form, an even product structure satisfying
$\mathsf T_1,\mathsf T_2,\mathsf T_4$ yields a zero-divisor counterexample over $\F_2$, and an odd
one yields a unit counterexample; if the structure moreover admits a Garg--Mineyev signature, in
the sense of \cite[Theorem~16]{GargMineyev}, the counterexample can be taken over any ring with
unity. The present paper does not use the internal definition of signature; it is only the
additional hypothesis in the external assembly theorem needed for this arbitrary-ring upgrade.
Moreover, the form of \cite[Theorem~16]{GargMineyev} stated with orientability and triple-girth is
equivalent there to the orientation/no-fold/triple-girth formulation $(2')$, and
Lemmas~\ref{lem:collapsed-middle} and~\ref{lem:gm-dictionary} identify the horizontal links,
middle link, no-fold condition, and triple-girth graphs with the collapsed objects used in this
paper. Thus obstructing the latter
obstructs every assembly input covered by that theorem. By Theorem~\ref{thm:main} no such product
structure exists, so each of these routes is closed.
\end{proof}

\begin{remark}[both values $4$ and $6$ are attained]\label{rem:example}
The dichotomy is sharp at both ends. \emph{The value $4$:} take $A=\{a_1,a_2,a_3\}$,
$B=\{b_1,b_2,b_3\}$, defect $(a_3,b_3)$, and the four $2$-cells
\[
\{(a_1,b_1),(a_2,b_2)\},\quad\{(a_1,b_2),(a_2,b_3)\},
\]
\[
\{(a_1,b_3),(a_3,b_1)\},\quad\{(a_2,b_1),(a_3,b_2)\}.
\]
These partition $A\times B$ minus the defect, and each cell differs in both coordinates. Orient
the $A$-edges $a_1\!\to\!a_2$, $a_2\!\to\!a_3$ (color $c_1$) and $a_1\!\to\!a_3$ (color $c_2$);
orientability then forces the $B$-edges $b_1\!\to\!b_2$, $b_2\!\to\!b_3$ (color $c_1$) and
$b_3\!\to\!b_1$ (color $c_2$). Both $L_A$ and $L_B$ are triangles. The no-fold condition holds at every vertex: $a_1$ has
$(c_1,\out),(c_2,\out)$; $a_2$ has $(c_1,\inn),(c_1,\out)$; $a_3$ has
$(c_1,\inn),(c_2,\inn)$; $b_1$ has $(c_1,\out),(c_2,\inn)$; $b_2$ has
$(c_1,\inn),(c_1,\out)$; and $b_3$ has $(c_1,\inn),(c_2,\out)$. In particular one
computes $\Sigma(a_2)=\Sigma(b_2)=\{(c_1,\out),(c_1,\inn)\}$, a $4$-cycle. So this is a valid
orientable no-fold product structure with $\girth(L_1)=4$ and $\girth(L_{AB})=3$. (Consistently
with $\girth(L_1)<6$, the Structure Theorem does not apply: the color profile here is
$\{(2,2),(1,1)\}$, summing to $5\ne 4$, because the rectangles overlap.)

\emph{The value $6$:} this is realized by Example~10 of \cite{GargMineyev}, the even
$(4,4)$ product structure on $A=\{a_1,\dots,a_4\}$, $B=\{b_1,\dots,b_4\}$ with the eight cells
\[
\begin{array}{ll}
\{(a_1,b_1),(a_2,b_2)\}, & \{(a_1,b_2),(a_3,b_3)\},\\
\{(a_2,b_1),(a_3,b_2)\}, & \{(a_1,b_3),(a_4,b_4)\},\\
\{(a_2,b_3),(a_4,b_1)\}, & \{(a_1,b_4),(a_3,b_1)\},\\
\{(a_2,b_4),(a_4,b_2)\}, & \{(a_4,b_3),(a_3,b_4)\}.
\end{array}
\]
Here $L_A$ and $L_B$ are both the complete graph $K_4$, so $\girth(L_{AB})=3$. The cells fall
into four colors $O,G,P,B$; the cell-compatible orientation must be recorded: the
ordered pairs of \cite{GargMineyev} specify the cells but not the coherent orientation used here. With the
orientation
\[
\begin{array}{c|c|c}
\text{color} & A\text{-edges} & B\text{-edges}\\ \hline
O & a_1\!\to\!a_2,\ a_2\!\to\!a_3 & b_1\!\to\!b_2\\
G & a_1\!\to\!a_3 & b_2\!\to\!b_3,\ b_4\!\to\!b_1\\
P & a_2\!\to\!a_4 & b_3\!\to\!b_1,\ b_4\!\to\!b_2\\
B & a_1\!\to\!a_4,\ a_4\!\to\!a_3 & b_3\!\to\!b_4
\end{array}
\]
(which is forced by the cells once $O$ is oriented as shown), the structure is no-fold, and the
profile is $\{(2,1),(1,2),(2,1),(1,2)\}$, so the product identity reads
$2\sum_c x_cy_c=2\cdot 8=16=mn$. Writing $\out,\inn$ as subscripts $o,i$, the signed-color sets
are
\[
\begin{aligned}
\Sigma(a_1)&=\{O_o,G_o,B_o\}, & \Sigma(a_2)&=\{O_i,O_o,P_o\},\\
\Sigma(a_3)&=\{O_i,G_i,B_i\}, & \Sigma(a_4)&=\{P_i,B_o,B_i\},\\
\Sigma(b_1)&=\{O_o,G_i,P_i\}, & \Sigma(b_2)&=\{O_i,G_o,P_i\},\\
\Sigma(b_3)&=\{G_i,P_o,B_o\}, & \Sigma(b_4)&=\{G_o,P_o,B_i\}.
\end{aligned}
\]
The displayed lists also verify no-fold at every vertex: each incident signed color appears at most
once, equivalently no two incident horizontal graph edges have the same color and the same direction.
No two of these eight sets meet in two signed colors, so $L_1$ has no $4$-cycle; and
$a_1-O_o-a_2-O_i-b_2-G_o-a_1$ is a $6$-cycle. Hence $\girth(L_1)=6$ directly. Thus the value
$6$ occurs, in accordance with the dichotomy's bound $\girth(L_{AB})\le 5$ in that case;
equivalently, the no-repeated-pattern condition $\mathsf T_3$ is satisfiable. Both ends of
Corollary~\ref{cor:dichotomy} are therefore non-vacuous, and the bound $\halfgirth(L_1)\le 3$
of Theorem~\ref{thm:A} is optimal. The value $6$ is moreover not tied to a triangle in the
horizontal graphs: Theorem~\ref{thm:affine} realizes $\girth(L_1)=6$ with $L_A\cong L_B\cong
K_{q,q}$, hence with $\girth(L_{AB})=4$.
\end{remark}

\begin{remark}[sharpness]
The estimates driving Theorem~\ref{thm:B} are tight against incidence geometry. The Moore
bound $E_m$ is sharp --- equality holds for the Levi graphs of finite projective planes --- and
the pressure--Fisher bound of Lemma~\ref{lem:pf} is nearly an equality at $m=6$: first
Lemma~\ref{lem:pf} forces $n=6$, then the girth-six Moore bound together with minimum degree $2$
forces $L_A=L_B=C_6$, and the remaining profile argument forces two colors of type $(3,3)$. The
extremal incidence structures the product identity competes against are thus the high-girth
Moore graphs, typified by projective-plane Levi graphs in the $(6,3)$ regime and by the
single-cell degeneracy of Theorem~\ref{thm:A} otherwise; the theorem asserts that an orientable
taiko can never reach that extreme. The marginality is genuine: bounding each $y_c$ by its
first-moment maximum gives only $n\lesssim m$, and it is the Fisher second moment in
Lemma~\ref{lem:pf} that turns the near-diagonal squeeze into a strict contradiction.
\end{remark}

\begin{remark}[necessity of pressure]\label{rem:necessary}
The pressure inequalities (vi) cannot be dropped: the product identity (iv) together with the
half-size bound (v) and the Moore bounds $e_A\le E_m$, $e_B\le E_n$ does \emph{not} by itself
yield a contradiction. Concretely, on the diagonal $m=n$ the bare constraints
$2\sum_c x_cy_c=mn-\delta$, $x_c,y_c\le h_m$, $\sum_c x_c\le E_m$, $\sum_c y_c\le E_m$ admit a
feasible assignment of color masses (ignoring whether it is realized by an actual structure):
for $m=n=2r$ even, take two colors with $(x_c,y_c)=(r,r)$, giving $\sum_c x_cy_c=2r^2=m^2/2$;
for $m=n=2r+1$ odd, take colors $(r+1,r)$ and $(r,r+1)$, giving
$\sum_c x_cy_c=2r(r+1)=(m^2-1)/2$. Both meet the identity exactly while respecting half-size and
$E_m\ge m$ (valid for $m\ge 6$). What forbids the construction is the finer content of (vi): a
color with $x_c\approx m/2$ is forced to be light on the other side, $y_c\lesssim\sqrt{2m}$,
rather than the $h_n\approx n/2$ a naive count would allow. Lemma~\ref{lem:pf} is precisely the
exploitation of this two-sided smallness: the convex chord converts the per-color pressure bound
into a global second-moment inequality that the product identity cannot satisfy.
\end{remark}

\begin{remark}[self-containment]
The proof of Theorem~\ref{thm:main} uses no computer search and no enumeration of product
structures. The only small-support inputs are short structural lemmas: the shapes with a class of
size at most three are eliminated by Lemmas~\ref{lem:m2}--\ref{lem:threethree} and
Proposition~\ref{prop:six}; all larger shapes are eliminated by the structural arguments of
Sections~\ref{sec:A} and~\ref{sec:B}. In particular, within the standing range
$m,n\ge2$, the prior exclusions of \cite{GargMineyev} for $2\le\min\{m,n\}\le5$ and for
$2\le m,n\le 13$ are recovered as special cases; the one-coordinate boundary cases are precisely
the degenerate cases excluded by the standing convention above. We draw from
\cite{GargMineyev} only the framework of Section~\ref{sec:taiko}---the definition of a product
structure, the conditions $\mathsf T_1$--$\mathsf T_4$, the assembly statement that a structure
meeting them yields a zero divisor in the even case or a nontrivial unit in the odd case over
$\F_2$, and the elementary incidence facts recalled there, each cited at its point of use; the
girth analysis of Sections~\ref{sec:structure}--\ref{sec:affine} is proved here and reuses none of
their exclusion arguments.
\end{remark}

\begin{remark}[coefficient rings]\label{rem:rings}
The dichotomy and Theorem~\ref{thm:main} are statements of pure combinatorics, with no field
appearing. Thus the obstruction to the same orientable no-fold $\mathsf T_4$ product-structure
data is coefficient-independent at the level of taikos. Applying the Garg--Mineyev assembly theorem
over a chosen coefficient ring still requires the corresponding \emph{signature} hypothesis, in the
sense of \cite[Theorem 16]{GargMineyev}; we use that term here only as the named additional
hypothesis in the external theorem. With that qualification, such a product structure yields a
counterexample over $\F_2$, and, if it admits the required signature, over \emph{any} ring with unity;
since Theorem~\ref{thm:main} shows no such product structure exists, every assembly theorem whose
input is this same data is obstructed independently of coefficients.
This does not rule out essentially different constructions --- ones not presented by an orientable
no-fold product structure, for instance a characteristic-sensitive assembly using cells whose sizes
track the additive combinatorics of the coefficients rather than the relation $1+1=0$ encoded by a
size-two cell. Our result obstructs this product-structure route; it says nothing about constructions
of a different shape.
\end{remark}

\section{Affine planes and the girth frontier}\label{sec:affine}

The dichotomy bounds the middle link of every such orientable no-fold product structure, and the
value $\girth(L_1)=6$ does occur (Remark~\ref{rem:example}). We now show that it occurs even
when the horizontal graphs are as far from carrying a triangle as the dichotomy allows: there
is an infinite affine-plane family with $\girth(L_1)=6$ whose horizontal graphs are complete
bipartite, hence triangle-free of girth four. This is the lower construction matching the
obstruction; together with the closure proved at the end of this section
(Theorem~\ref{thm:p5}) it pins the horizontal frontier at $p^\ast=5$.

\begin{theorem}[affine-plane sharpness]\label{thm:affine}
Let $q\ge 4$ be a power of $2$. There is an even orientable no-fold product structure of type
$(2q,2q)$ with
\[
        \girth(L_1)=6,\qquad L_A\cong L_B\cong K_{q,q}.
\]
In particular $\halfgirth(L_1)=3$ and $\girth(L_{AB})=4$, realizing the frontier pair $(4,3)$.
\end{theorem}

\begin{proof}
Work in the affine plane $\mathrm{AG}(2,q)$ over $\F=\F_q$ \cite{HughesPiper}: its points are the
vectors of $\F^2$, and its lines fall into $q+1$ parallel classes (\emph{directions}), one for each
slope in $\F\cup\{\infty\}$, each class consisting of $q$ disjoint lines that cover all $q^2$ points.
Since $q\ge 4$ there are at least five directions; fix four distinct ones
$\alpha,\beta,\gamma,\omega$, and write $\ell_\rho(P)$ for the unique line of direction $\rho$
through a point $P$. Put
\[
        A=\{\text{lines of direction }\alpha\text{ or }\beta\},\qquad
        B=\{\text{lines of direction }\gamma\text{ or }\omega\},
\]
so $\abs A=\abs B=2q$, and set $\bar\alpha=\beta$, $\bar\beta=\alpha$, $\bar\gamma=\omega$,
$\bar\omega=\gamma$.

Fix a nonzero $v\in\F^2$ whose direction is none of $\alpha,\beta,\gamma,\omega$ (possible since
$q\ge4$ leaves at least one further direction); then $P$ and $P+v$ lie on distinct lines of each
of the four selected directions. As $q$ is even, $\tau(P)=P+v$ is a fixed-point-free involution of
$\F^2$ (because $2v=0$ and $v\ne0$); this is the only point in the construction where
characteristic two is used. Let $R$ be a transversal of its orbits $\{P,P+v\}$, so
$\abs R=q^2/2$. For each $P\in R$, each $i\in\{\alpha,\beta\}$, and each $j\in\{\gamma,\omega\}$
form the $2$-cell
\[
        \kappa(P,i,j)=\bigl\{\,(\ell_i(P),\,\ell_j(P)),\ (\ell_{\bar i}(P+v),\,\ell_{\bar j}(P+v))\,\bigr\}.
\]
There are $\abs R\cdot 4=2q^2$ such cells, half of $\abs{A\times B}=4q^2$ (Figure~\ref{fig:affine}).

\emph{Partition and disjoint vertices.} Let $(L,M)\in A\times B$, with $L$ of direction
$i\in\{\alpha,\beta\}$ and $M$ of direction $j\in\{\gamma,\omega\}$. Since $i\ne j$, the lines
$L,M$ meet in a unique point $Q$, and $L=\ell_i(Q)$, $M=\ell_j(Q)$. If $Q\in R$, then $(L,M)$ is
the first coordinate of $\kappa(Q,i,j)$; if $Q\notin R$, then $Q+v\in R$, and since $(Q+v)+v=Q$
and the bar operation is an involution, $(L,M)$ is the second coordinate of $\kappa(Q+v,\bar i,\bar j)$. The
point $Q$ and the directions $i,j$ are recovered from $(L,M)$, so this occurrence is unique;
hence the cells partition $A\times B$. The two $A$-lines of a cell have directions $i\ne\bar i$
and the two $B$-lines directions $j\ne\bar j$, so each cell differs in both coordinates.

\emph{Horizontal graphs.} The cell $\kappa(P,i,j)$ induces the $A$-edge
$\{\ell_i(P),\ell_{\bar i}(P+v)\}$ and the $B$-edge $\{\ell_j(P),\ell_{\bar j}(P+v)\}$; thus
every horizontal edge joins an $\alpha$-line to a $\beta$-line (on the $A$-side) or a
$\gamma$-line to an $\omega$-line (on the $B$-side), with no edge inside a direction class.
For an affine line $N$ and vector $u$, write $N-u=\{X-u:X\in N\}$. Conversely, let $L$ be an $\alpha$-line and $M$ a $\beta$-line. The translate $M-v$ is again a
$\beta$-line, so $L\cap(M-v)$ is a single point $P_0$; then $P_0\in L$ and $P_0+v\in M$. If
$P_0\in R$, the cells with representative $P_0$ supply the unordered edge
$\{L,M\}=\{\ell_\alpha(P_0),\ell_\beta(P_0+v)\}$; if $P_0\notin R$, the cells with representative
$P_0+v$ supply the same unordered edge. Hence every
$\alpha$-line is joined to every $\beta$-line, so $L_A\cong K_{q,q}$, and the directions
$\gamma,\omega$ give $L_B\cong K_{q,q}$ identically. The $2q^2$ cells contribute $q^2$ distinct
$A$-edges: for fixed $P$ the two choices of $j\in\{\gamma,\omega\}$ induce the same $A$-edge
$\{\ell_i(P),\ell_{\bar i}(P+v)\}$, while distinct $(P,i)$ give distinct edges by the
intersection count just made (and the two $A$-lines of a fixed $P$ differ, since $P$ and $P+v$
lie on distinct $\alpha$-lines). These $q^2$ edges are exactly $\abs{E(K_{q,q})}$; so each edge
of the simple graph $L_A$ arises from a cell, and $\girth(L_{AB})=4$ as $q\ge 2$.

\emph{Colors, orientation, no-fold.} For $P\in R$ the four cells $\kappa(P,\cdot,\cdot)$ link
the two $A$-edges $\{\ell_\alpha(P),\ell_\beta(P+v)\}$, $\{\ell_\beta(P),\ell_\alpha(P+v)\}$ to
the two $B$-edges in the complete bipartite pattern (each of the two $A$-edges shares a cell with
each of the two $B$-edges), so the taiko equivalence relation of Section~\ref{sec:taiko} merges
these four horizontal edges into a single class: the derived color $c_P$ of all four cells, with
$x_{c_P}=y_{c_P}=2$. Distinct representatives give distinct classes: an unordered $A$-edge
between an $\alpha$-line $L$ and a $\beta$-line $M$ determines the orbit
$\{L\cap(M-v),\,M\cap(L-v)\}=\{Q,Q+v\}$, and hence its unique representative in $R$; the same
argument applies to the $B$-edges. Therefore no generator of the color equivalence relation coming
from a cell with representative $P$ shares a horizontal edge with a cell of representative
$P'\ne P$. Thus the $q^2/2$ classes partition the edges of $L_A$ and of $L_B$ (whence
$2\sum_c x_cy_c=4q^2=mn$). Orient the color-$c_P$ edges by $\ell_i(P)\to\ell_{\bar i}(P+v)$ on the $A$-side and
$\ell_j(P)\to\ell_{\bar j}(P+v)$ on the $B$-side. This is well defined because the classes are
edge-disjoint, and in each cell $\kappa(P,i,j)$ the first coordinate is a pair of tails and the
second a pair of heads; the structure is therefore orientable. For no-fold, fix $c_P$: on the
$A$-side its tails are the two $A$-lines through $P$ and its heads the two $A$-lines through
$P+v$, so a line is a tail of at most one $c_P$-edge and a head of at most one, and no line is
two tails or two heads. Since distinct representatives give distinct colors, no-fold only forbids
interactions among edges of the fixed color $c_P$ just checked; the $B$-side is identical.

\emph{Middle link.} Identify signed colors with affine points as follows. For each representative
$P\in R$, assign $(c_P,\out)$ to the point $P$ and $(c_P,\inn)$ to the point $P+v$. Thus every
point $Q\in\F^2$ labels exactly one signed color: if $Q\in R$ it labels $(c_Q,\out)$, while if
$Q\notin R$ it labels $(c_{Q+v},\inn)$, since $Q+v\in R$. The lines incident in $L_1$ to the
signed color labelled by $Q$ are precisely
$\ell_\alpha(Q),\ell_\beta(Q),\ell_\gamma(Q),\ell_\omega(Q)$, the four selected lines through
$Q$. Under this identification, the signed-color side of $L_1$ is the affine-point side, while
$A\sqcup B$ is the selected-line side. Thus $L_1$ is the point--line incidence graph of
$\mathrm{AG}(2,q)$ restricted to the lines of directions $\alpha,\beta,\gamma,\omega$. Two affine points lie on at most one
common line and two affine lines meet in at most one point, so $L_1$ has no $4$-cycle and
$\girth(L_1)\ge 6$. It has a $6$-cycle: let $O$ be the origin, choose $P_1\in\ell_\alpha(O)$
with $P_1\ne O$, set $L_\gamma=\ell_\gamma(P_1)$ and $P_2=\ell_\beta(O)\cap L_\gamma$. The three
selected lines $\ell_\alpha(O),\ell_\beta(O),L_\gamma$ meet pairwise in the points $O,P_1,P_2$,
which are distinct (the line through any two of them has direction $\alpha$, $\beta$, or
$\gamma$, while the third point lies off that line), so
\[
        O-\ell_\alpha(O)-P_1-L_\gamma-P_2-\ell_\beta(O)-O
\]
is a $6$-cycle in $L_1$. Hence $\girth(L_1)=6$.

The structure is even, orientable, and no-fold, with $\girth(L_1)=6$ and
$L_A\cong L_B\cong K_{q,q}$; thus $\halfgirth(L_1)=3$ and $\girth(L_{AB})=4$.
\end{proof}

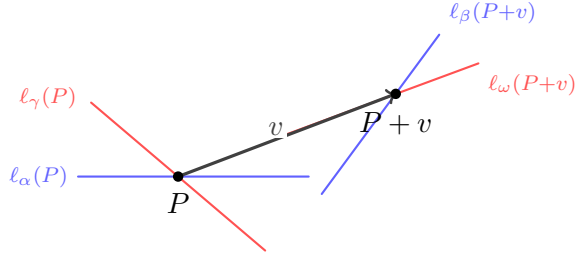
\begin{figure}[t]
\centering
\begin{tikzpicture}[scale=1.15,line cap=round]
  \coordinate (P) at (0,0);
  \coordinate (Q) at (2.5,0.95);
  \draw[blue!62,thick] (-1.15,0) -- (1.5,0);
  \draw[blue!62,thick] (1.65,-0.20) -- (3.0,1.63);
  \draw[red!68,thick] (-1.0,0.85) -- (1.0,-0.85);
  \draw[red!68,thick] (1.15,0.45) -- (3.45,1.30);
  \draw[->,very thick,black!75] (P) -- (Q) node[pos=0.45,above,fill=white,inner sep=1pt]{$v$};
  \fill (P) circle (1.7pt) node[below=2.5pt]{$P$};
  \fill (Q) circle (1.7pt) node[below=3pt]{$P+v$};
  \node[blue!62,font=\scriptsize,anchor=east] at (-1.15,0) {$\ell_\alpha(P)$};
  \node[red!68,font=\scriptsize,anchor=east] at (-1.05,0.9) {$\ell_\gamma(P)$};
  \node[blue!62,font=\scriptsize,anchor=south west] at (3.0,1.63) {$\ell_\beta(P{+}v)$};
  \node[red!68,font=\scriptsize,anchor=north west] at (3.45,1.30) {$\ell_\omega(P{+}v)$};
\end{tikzpicture}
\caption{The cell $\kappa(P,\alpha,\gamma)$ of the affine construction. The signed color $c_P$ is
indexed by the size-two orbit $\{P,P+v\}$ of the fixed-point-free involution $\tau(P)=P+v$, an
involution exactly because $2v=0$, i.e.\ $\operatorname{char}\F=2$. Its two $A$-lines
$\ell_\alpha(P),\ell_\beta(P+v)$ (directions $\alpha,\beta$) form an edge of $L_A$, and its two
$B$-lines $\ell_\gamma(P),\ell_\omega(P+v)$ an edge of $L_B$; the four cells $\kappa(P,\cdot,\cdot)$
assemble into the color class $c_P$.}
\label{fig:affine}
\end{figure}

\begin{remark}[on the characteristic]
The hypothesis that $q$ is a power of two enters only through the fixed-point-free involution
$\tau(P)=P+v$: as $2v=0$ and $v\ne0$, it pairs the $q^2$ affine points into the $q^2/2$ size-two
orbits $\{P,P+v\}$ that index the cells and the balanced signed colors $(c_P,\out)=P$,
$(c_P,\inn)=P+v$ with $x_c=y_c=2$. Over a plane of odd order the map $P\mapsto P+v$ has orbits of
size $\operatorname{char}\F>2$, so this pairing has no direct analogue. Whether the pair
$(\girth(L_{AB}),\halfgirth(L_1))=(4,3)$ is realized over odd-order planes by another mechanism we
leave open; it does not affect the frontier value, which Theorem~\ref{thm:p5} fixes at
$p^\ast=5$ regardless.
\end{remark}

The construction completes the picture of the horizontal frontier left by the dichotomy.

\begin{corollary}[the horizontal frontier]\label{cor:frontier}
Let
\[
\begin{aligned}
        p^\ast=\min\{\,p\in\mathbb Z_{\ge3}:&\ \nexists\,\Pi\text{ orientable no-fold, in the standing sense, with}\\
        &\ \girth(L_{AB}(\Pi))\ge p\text{ and }\halfgirth(L_1(\Pi))\ge 3\,\}.
\end{aligned}
\]
Then $p^\ast\in\{5,6\}$.
\end{corollary}

\begin{proof}
Theorem~\ref{thm:affine} realizes $\girth(L_{AB})=4$ with $\halfgirth(L_1)=3$, so the
obstruction fails at $p=4$, giving $p^\ast\ge 5$. Theorem~\ref{thm:B} rules out
$\girth(L_{AB})\ge 6$ together with $\halfgirth(L_1)\ge 3$ (the $(6,3)$ branch), so
$p^\ast\le 6$.
\end{proof}

\begin{lemma}[Moore bound, girth five]\label{lem:moore5}
For a positive integer $r$, write $M(r)=\lfloor\tfrac r2\sqrt{r-1}\rfloor$. Every finite simple graph of girth at least $5$
on $r$ vertices has at most $M(r)$ edges.
\end{lemma}

\begin{proof}
For $r=1$ the claim is trivial. Assume $r\ge2$, and write $\bar d=2|E|/r$ for the average degree.
If $\bar d\ge 2$, the irregular Moore bound of Alon,
Hoory, and Linial \cite{AlonHooryLinial} applies: a graph of girth at least $5$ and average degree
$\bar d\ge 2$ has at least $1+\bar d^{\,2}$ vertices, so $\bar d\le\sqrt{r-1}$, and
$|E|=\tfrac r2\bar d\le\tfrac r2\sqrt{r-1}$, whence $|E|\le M(r)$. If instead $\bar d<2$, then
$|E|<r$, so $|E|\le r-1\le M(r)$, the last step because $M(r)=\lfloor\tfrac r2\sqrt{r-1}\rfloor\ge
r-1$ for every $r\ge 2$. Either way $|E|\le M(r)$, with no hypothesis on the minimum degree.
\end{proof}

\begin{proposition}[near-Moore reduction at the frontier]\label{prop:nearmoore}
Let $(\Pi_i)$ be any sequence of orientable no-fold product structures, even or odd in the
standing taiko sense, with $\girth(L_1)\ge 6$ and
$\girth(L_A),\girth(L_B)\ge 5$. Write the type of $\Pi_i$ as $(m_i,n_i)$ and, after interchanging
sides if necessary for each $i$, assume $m_i\le n_i$. If $m_i\to\infty$, then, writing
$m=m_i$ and $n=n_i$ along the sequence,
\[
\begin{gathered}
        n_i=(1+o(1))m_i,\qquad e_A(\Pi_i)=(\tfrac12+o(1))m_i^{3/2},\\
        e_B(\Pi_i)=(\tfrac12+o(1))n_i^{3/2}.
\end{gathered}
\]
Thus any unbounded family realizing the intermediate pair $(5,3)$ is asymptotically diagonal,
with both horizontal graphs Moore-extremal for girth five.
\end{proposition}

\begin{proof}
We suppress the index $i$ throughout the proof; all asymptotic notation refers to the given
sequence. By the girth-five Moore bound, Lemma~\ref{lem:moore5},
\[
        e_A\le M(m)\le \tfrac m2\sqrt{m-1},\qquad
        e_B\le M(n)\le \tfrac n2\sqrt{n-1}.
\]
Write $\widetilde M_m=\tfrac m2\sqrt{m-1}$ and $\widetilde M_n=\tfrac n2\sqrt{n-1}$.

The pressure inequality~(vi) gives $y_c(m-x_c)\le 2e_A-x_c+\delta\le 2e_A+\delta$, and since
$x_c\le h_m\le\tfrac{m+1}2$ we have $m-x_c\ge\tfrac{m-1}2$, so
\[
        y_c\le\frac{2(2e_A+\delta)}{m-1}\le\frac{4\widetilde M_m+2\delta}{m-1}=(2+o(1))\sqrt m
        \qquad(\text{all }c).
\]
Thus
\[
        \frac{mn-\delta}{2}=\sum_c x_cy_c
        \le (2+o(1))\sqrt m\sum_c x_c
        =(2+o(1))\sqrt m\,e_A
        \le (1+o(1))m^2,
\]
which gives $n\le(2+o(1))m$; since $n\ge m$, we have $n=O(m)$.

Now use the dual pressure inequality with this crude diagonal control. Because $y_c=O(\sqrt m)=o(n)$,
\[
        x_c(n-y_c)\le 2e_B-y_c+\delta\le 2e_B+\delta
\]
gives
\[
        x_c\le \frac{2e_B+\delta}{n-y_c}
        \le \frac{2\widetilde M_n+\delta}{(1-o(1))n}=(1+o(1))\sqrt n=O(\sqrt m)=o(m).
\]
Consequently $m-x_c=(1-o(1))m$. Reapplying pressure on the $A$-side with this sharper denominator,
\[
        y_c\le\frac{2e_A+\delta}{m-x_c}
        \le\frac{2\widetilde M_m+\delta}{(1-o(1))m}=(1+o(1))\sqrt m
        \qquad(\text{all }c).
\]
The product identity then gives
\[
        \frac{mn-\delta}{2}=\sum_c x_cy_c
        \le(1+o(1))\sqrt m\,e_A
        \le(\tfrac12+o(1))m^2,
\]
so $n\le(1+o(1))m$; together with $n\ge m$ this is $n=(1+o(1))m$. The same displayed inequality,
now read as a lower bound on $e_A$, gives
$e_A\ge(\tfrac12-o(1))m^{3/2}$, and the Moore upper bound gives
$e_A=(\tfrac12+o(1))m^{3/2}$.

It remains only to justify the symmetric sharp lower bound for $e_B$, rather than assume it. We
already know $y_c=O(\sqrt m)=o(n)$, and the dual pressure estimate above sharpens, using
$e_B\le M(n)\le \widetilde M_n$, to
\[
        x_c\le \frac{2e_B+\delta}{n-y_c}
        \le \frac{2\widetilde M_n+\delta}{(1-o(1))n}=(1+o(1))\sqrt n
        \qquad(\text{all }c).
\]
Therefore
\[
        \frac{mn-\delta}{2}=\sum_c x_cy_c
        \le(1+o(1))\sqrt n\sum_c y_c
        =(1+o(1))\sqrt n\,e_B.
\]
Since $m=(1+o(1))n$, the left side is $(\tfrac12+o(1))n^2$, whence
$e_B\ge(\tfrac12-o(1))n^{3/2}$; the Moore upper bound gives
$e_B=(\tfrac12+o(1))n^{3/2}$. Combining the sharp bounds for $x_c$ and $y_c$ also gives the
uniform width bound claimed after the proposition.
\end{proof}

The proof of Proposition~\ref{prop:nearmoore} also yields $x_c,y_c\le(1+o(1))\sqrt m$ uniformly,
hence $W:=\max_s\max(|A_s|,|B_s|)\le(1+o(1))\sqrt m$. This last bound is the lever that upgrades
the reduction to an outright nonexistence statement. The mechanism is a \emph{weighted} form of
the dual Fisher inequality, in which each pair of signed colors is charged by the sizes on the
side where it does \emph{not} meet.

All sums indexed by $\{s,t\}$ below are over unordered two-element subsets of distinct signed
colors.
\begin{lemma}[weighted dual Fisher]\label{lem:wdf}
For any orientable no-fold product structure, even or odd in the standing taiko sense, with
$\girth(L_1)\ge 6$, writing $\alpha_s=|A_s|$ and
$\beta_s=|B_s|$ over the signed colors $s$,
\begin{multline*}
        \tfrac12\Bigl(mn^2-\delta(2n-1)-\sum_s\alpha_s\beta_s^2\Bigr)
       +\tfrac12\Bigl(nm^2-\delta(2m-1)-\sum_s\beta_s\alpha_s^2\Bigr)\\
        \le\ \sum_{\{s,t\}}\max(\alpha_s\alpha_t,\ \beta_s\beta_t).
\end{multline*}
\end{lemma}

\begin{proof}
Fix $a\in A$. By the rectangle partition in the Structure Theorem, equivalently by the proof of
the tiling identity~(i), the sets $B_s$ for $s\in\Sigma(a)$ partition $B$ minus the defect. Hence
$\sum_{s\in\Sigma(a)}\beta_s=n-\one_{a\in D_A}$ and
\[
        \sum_{\{s,t\}\subset\Sigma(a)}\beta_s\beta_t
        =\tfrac12\Bigl((n-\one_{a\in D_A})^2-\sum_{s\in\Sigma(a)}\beta_s^2\Bigr).
\]
Summing over $a\in A$ and using $\sum_a(n-\one_{a\in D_A})^2=mn^2-\delta(2n-1)$ together with
$\sum_a\sum_{s\in\Sigma(a)}\beta_s^2=\sum_s\alpha_s\beta_s^2$ gives the first term on the left;
the $B$-side gives the second. Reorganizing the same sum by pairs, the left side equals
\[
        \sum_{\{s,t\}}\bigl(\beta_s\beta_t\,|A_s\cap A_t|+\alpha_s\alpha_t\,|B_s\cap B_t|\bigr).
\]
By near-disjointness~(ii) each pair has $|A_s\cap A_t|+|B_s\cap B_t|\le 1$, so contributes at
most one of $\beta_s\beta_t,\ \alpha_s\alpha_t$, in either case at most
$\max(\alpha_s\alpha_t,\beta_s\beta_t)$.
\end{proof}

Ordinary dual Fisher charges every pair the same; the weighted form charges a pair
$\{s,t\}$ only by $\max(\alpha_s\alpha_t,\beta_s\beta_t)$, so a light signed color
($\alpha_s,\beta_s$ small) barely inflates the right-hand side. In the near-Moore regime the
left-hand side is forced to be twice as large as the right, which is impossible.

\begin{theorem}[asymptotic closure of the $(5,3)$ frontier]\label{thm:asymp53}
There is an absolute constant $M_0$ such that every orientable no-fold product structure, even or
odd in the standing taiko sense, with $\girth(L_1)\ge 6$ and $\girth(L_A),\girth(L_B)\ge 5$ has $\min\{m,n\}\le M_0$. For each fixed
value of $\min\{m,n\}$ the pressure inequality bounds $\max\{m,n\}$, so only finitely many types
occur; in particular no infinite family realizes the pair $(5,3)$.
\end{theorem}

\begin{proof}
If no such $M_0$ existed, then $\min\{m,n\}$ would be unbounded over the family. Passing to a
subsequence and interchanging the two sides when necessary, we may assume
\[
        m=\min\{m,n\}\to\infty,\qquad m\le n .
\]
Write $w_s=\max(\alpha_s,\beta_s)$ and $W=\max_s w_s$. By
Proposition~\ref{prop:nearmoore},
\[
        n=(1+o(1))m,\qquad
        e_A,e_B=(\tfrac12+o(1))m^{3/2},\qquad
        W\le(1+o(1))\sqrt m .
\]

\emph{Left side of Lemma~\ref{lem:wdf}.} The leading terms are
$\tfrac12(mn^2+nm^2)=\tfrac12mn(m+n)=(1+o(1))m^3$. The subtracted terms are lower order:
$\sum_s\alpha_s\beta_s^2\le W\sum_s\alpha_s\beta_s=W(mn-\delta)=O(m^{5/2})$, and likewise for
$\sum_s\beta_s\alpha_s^2$. Hence the left side equals $(1+o(1))m^3$.

\emph{Right side.} Since $\max(\alpha_s\alpha_t,\beta_s\beta_t)\le w_sw_t$,
\[
        \sum_{\{s,t\}}\max(\alpha_s\alpha_t,\beta_s\beta_t)\ \le\ \tfrac12\Bigl(\sum_s w_s\Bigr)^2 .
\]
Now $\sum_s w_s=\sum_s\bigl(\alpha_s+\beta_s-\min(\alpha_s,\beta_s)\bigr)
=2e_A+2e_B-\sum_s\min(\alpha_s,\beta_s)$, and since
$\min(\alpha_s,\beta_s)=\alpha_s\beta_s/w_s\ge\alpha_s\beta_s/W$,
\[
        \sum_s\min(\alpha_s,\beta_s)\ \ge\ \frac{1}{W}\sum_s\alpha_s\beta_s=\frac{mn-\delta}{W}
        \ \ge\ (1-o(1))m^{3/2},
\]
while $2e_A+2e_B=(2+o(1))m^{3/2}$. Therefore $\sum_s w_s\le(1+o(1))m^{3/2}$, and the right side
is at most $(\tfrac12+o(1))m^3$.

Combining the two estimates, $(1+o(1))m^3\le(\tfrac12+o(1))m^3$, which is false for large $m$;
hence $\min\{m,n\}\le M_0$. Finally, after interchanging $A$ and $B$ if necessary, fix
$m=\min\{m,n\}\le n$. For every color, pressure and the half-size bound give
\[
        y_c\le\frac{2e_A-x_c+\delta}{m-x_c}
        \le\frac{2M(m)+\delta}{m-h_m}
        \le\frac{4M(m)+2\delta}{m-1},
\]
where $h_m=\lfloor(m+1)/2\rfloor$ and $m-h_m=\lfloor m/2\rfloor\ge(m-1)/2$. This is a bound
depending only on $m$. With at most $e_A\le M(m)$ colors and each $x_c\le h_m$, the product identity
then forces $\tfrac{mn-\delta}2=\sum_c x_cy_c=O_m(1)$, so $\max\{m,n\}=n$ is bounded in terms of
$m$.
\end{proof}

The bound $W\le(1+o(1))\sqrt m$ is essential: without it the weighted inequality is satisfiable,
and indeed an ``anti-correlated'' profile with $e_A$ a constant factor above the girth-five
Moore bound evades it. It is precisely the near-Moore squeeze of
Proposition~\ref{prop:nearmoore} --- forcing every signed color to be thin --- that makes the
left side win.

\subsection{From asymptotic to exact: closing the frontier}

Theorem~\ref{thm:asymp53} leaves only a bounded sporadic remainder. We eliminate it by a finite,
auditable verification at the level of signed-color \emph{profiles}---the multiset of pairs
$(x_c,y_c)$ over active colors---without ever enumerating product structures. We first record the
constraints a profile must satisfy, then reduce to an explicit short list, then clear the list.
Throughout we may restrict to $m,n\ge 5$:

\begin{lemma}[support at the frontier]\label{lem:fivesupport}
Let $\Pi$ be an orientable no-fold product structure, even or odd in the standing taiko sense,
with $\girth(L_1)\ge 6$ and $\girth(L_A),\girth(L_B)\ge 5$, of type $(m,n)$ with $m\le n$.
Then $m,n\ge 5$.
\end{lemma}

\begin{proof}
Lemma~\ref{lem:m2} gives $m\ge 3$. By Lemma~\ref{lem:degreeone}, a vertex of $L_A$ of degree one forces $n\le 3$ and, symmetrically,
a degree-one vertex of $L_B$ forces $m\le 3$; and no vertex is isolated, since every row and
column meets a $2$-cell. Since $L_{AB}=L_A\sqcup L_B$ and $\girth(L_A),\girth(L_B)\ge 5$, we have
$\girth(L_{AB})\ge 5\ge 4$.

If $m=3$, then $L_A$ has girth at least five on three vertices, hence is acyclic; having no
isolated vertex it has a vertex of degree one, so $n\le 3$ and $\Pi$ has type $(3,3)$, excluded
by Lemma~\ref{lem:threethree}. If $m=4$, then $L_A$ has no $3$- or $4$-cycle and is therefore a
forest; again it has a degree-one vertex, forcing $n\le 3$, contrary to $m\le n$. Hence
$m\ge 5$, and so $n\ge 5$.
\end{proof}

\begin{proposition}[profile constraints]\label{prop:profile}
Let $\Pi$ be an orientable no-fold product structure, even or odd in the standing taiko sense,
with $\girth(L_1)\ge 6$, $\girth(L_A),\girth(L_B)\ge 5$, type $(m,n)$, defect
$\delta\in\{0,1\}$, and $m\le n$. Then
$\delta\equiv mn\pmod 2$, and the active colors give positive integers $(x_c,y_c)$ with
$x_c\le h_m$, $y_c\le h_n$ obeying
\begin{itemize}\itemsep2pt
\item product identity: $\sum_c x_cy_c=\tfrac{mn-\delta}2$;
\item girth-five Moore: $e_A:=\sum_c x_c\le M(m)$ and $e_B:=\sum_c y_c\le M(n)$;
\item Fisher: $\sum_c x_c(x_c-1)\le\binom m2$ and $\sum_c y_c(y_c-1)\le\binom n2$;
\item pressure: $y_c(m-x_c)\le 2e_A-x_c+\delta$ and $x_c(n-y_c)\le 2e_B-y_c+\delta$ for each $c$;
\item weighted dual Fisher, in the profile form below, where the final sum is over unordered pairs of distinct active colors,
\begin{multline}\label{eq:wdfprofile}
        \tfrac12\bigl(mn^2-\delta(2n-1)-2\textstyle\sum_c x_cy_c^2\bigr)
       +\tfrac12\bigl(nm^2-\delta(2m-1)-2\sum_c y_cx_c^2\bigr)\\
        \le\ \sum_c\max(x_c^2,y_c^2)\ +\ 4\!\!\sum_{\{c,c'\}}\!\max(x_cx_{c'},\,y_cy_{c'}).
\end{multline}
\end{itemize}
\end{proposition}

\begin{proof}
The parity is forced by $2\sum_c x_cy_c=mn-\delta$, item~(iv) of the Structure Theorem. The Moore
bounds are Lemma~\ref{lem:moore5} applied to the girth-five graphs $L_A,L_B$, whose edge counts
are $e_A=\sum_c x_c$ and $e_B=\sum_c y_c$. Fisher and pressure are items~(vii) and~(vi).
For~\eqref{eq:wdfprofile}, each color $c$ contributes two signed colors of common size
$(x_c,y_c)$; substituting in Lemma~\ref{lem:wdf} gives $\sum_s\alpha_s\beta_s^2=2\sum_c x_cy_c^2$
and $\sum_s\beta_s\alpha_s^2=2\sum_c y_cx_c^2$ on the left, while on the right the two signed
colors of one color contribute $\max(x_c^2,y_c^2)$ and each pair of distinct colors contributes
$4\max(x_cx_{c'},y_cy_{c'})$.
\end{proof}

\begin{proposition}[computer-assisted reduction to a finite list]\label{prop:reduce53}
If the constraints of Proposition~\ref{prop:profile} are simultaneously satisfiable for some
$(m,n,\delta)$ with $5\le m\le n$, $\delta\in\{0,1\}$, and $\delta\equiv mn\pmod2$, then
$(m,n,\delta)$ belongs to the explicit set $\mathcal F$ of the
$24$ triples
\begin{align*}
\mathcal F=\{&(5,5,1),(5,7,1),(7,7,1),(7,8,0),(7,9,1),(8,8,0),(8,9,0),(8,10,0),\\
&(9,9,1),(9,10,0),(9,11,1),(9,13,1),(10,10,0),(10,14,0),(10,15,0),\\
&(11,11,1),(11,13,1),(11,15,1),(13,13,1),(13,15,1),(13,17,1),\\
&(14,14,0),(14,15,0),(15,15,1)\},
\end{align*}
each with $\min\{m,n\}\le 15$ and $\max\{m,n\}\le 17$.
\end{proposition}

\begin{proof}
The hypothesis gives $5\le m\le n$. Put $w_c=\max(x_c,y_c)$,
$W=\max_c w_c$, $\Lambda=mn-\delta$, and
$C=\tfrac12mn(m+n)-\delta(m+n-1)$. Bounding the left side of~\eqref{eq:wdfprofile} below by
$x_c,y_c\le W$ gives $\sum_c x_cy_c(x_c+y_c)\le W\Lambda$, so the left side is at least $C-W\Lambda$. On the
right, $\max(x_c^2,y_c^2)=w_c^2$ and $\max(x_cx_{c'},y_cy_{c'})\le w_cw_{c'}$, so the right side
is at most $2(\sum_c w_c)^2$; and since $\sum_c\min(x_c,y_c)=\sum_c x_cy_c/w_c\ge \Lambda/(2W)$,
\[
        \sum_c w_c=e_A+e_B-\sum_c\min(x_c,y_c)\le e_A+e_B-\frac{\Lambda}{2W}.
\]
Hence~\eqref{eq:wdfprofile} forces the scalar inequality
\begin{equation}\label{eq:scalarwdf}
        C-W\Lambda\ \le\ \tfrac12\bigl(2e_A+2e_B-\tfrac{\Lambda}{W}\bigr)^2 ,
\end{equation}
and the product identity adds, with $\bar y=\max_c y_c$ and $\bar x=\max_c x_c$,
\begin{equation}\label{eq:reach}
        \tfrac{\Lambda}{2}=\sum_c x_cy_c\le\bar y\,e_A\quad\text{and}\quad\tfrac{\Lambda}{2}\le\bar x\,e_B .
\end{equation}
For a budget pair $(E_A,E_B)$, call a pair $(x,y)$ 
\emph{pressure-admissible at $(E_A,E_B)$} if
\[
\begin{gathered}
        1\le x\le h_m,\qquad 1\le y\le h_n,\\
        y(m-x)\le2E_A-x+\delta,
        \qquad
        x(n-y)\le2E_B-y+\delta .
\end{gathered}
\]
Let $W(E_A,E_B)$, $\bar x(E_A,E_B)$, and $\bar y(E_A,E_B)$ denote, respectively, the maxima of
$\max(x,y)$, of $x$, and of $y$ over all pairs pressure-admissible at $(E_A,E_B)$ (if there is no
admissible pair, no profile exists and the triple is already excluded). These are maxima over the
full admissible cell set, not over a particular profile.

\begin{claim}[corner monotonicity]
For fixed $(m,n,\delta)$, suppose~\eqref{eq:scalarwdf} and~\eqref{eq:reach} hold for some profile
with edge budgets $e_A\le M(m)$ and $e_B\le M(n)$. Then the same two inequalities hold at the
Moore corner $(E_A,E_B)=(M(m),M(n))$ when $W,\bar x,\bar y$ are replaced by the all-admissible
maxima $W(E_A,E_B),\bar x(E_A,E_B),\bar y(E_A,E_B)$.
\end{claim}

\begin{proof}[Proof of claim]
The profile's colors are pressure-admissible at their actual budgets $(e_A,e_B)$, so replacing the
profile maxima by the all-admissible maxima at $(e_A,e_B)$ can only increase $W,\bar x,\bar y$.
The reachability inequalities~\eqref{eq:reach} then only relax, since their right sides increase.
In~\eqref{eq:scalarwdf}, the left side $C-W\Lambda$ decreases as $W$ grows. For the right side, the
reachability inequalities and $\bar x,\bar y\le W$ give
\[
        \frac{\Lambda}{2}\le \bar y e_A\le W e_A,
        \qquad
        \frac{\Lambda}{2}\le \bar x e_B\le W e_B,
\]
hence $\Lambda/W\le2\min(e_A,e_B)$ and
$2e_A+2e_B-\Lambda/W\ge2\max(e_A,e_B)>0$. The bracket is nondecreasing in each of $e_A,e_B,W$ on this
admissible range, because $\Lambda/W$ decreases when $W$ grows; being nonnegative, its square is also
nondecreasing. Thus~\eqref{eq:scalarwdf} also relaxes when one replaces profile maxima by
all-admissible maxima at the same budget.

Finally, as $(E_A,E_B)$ increase, the pressure-admissible cell set only grows, so
$W(E_A,E_B)$, $\bar x(E_A,E_B)$, and $\bar y(E_A,E_B)$ are nondecreasing. The preceding
monotonicity therefore shows that both~\eqref{eq:scalarwdf} and~\eqref{eq:reach} continue to hold
as the budgets are raised from $(e_A,e_B)$ to the Moore corner $(M(m),M(n))$.
\end{proof}

Consequently, a triple for which the corner test fails supports no profile satisfying the full
constraints, and it suffices to test that corner.

When $n$ is large relative to $m$ the corner test~\eqref{eq:reach} already fails: since
$R(x)=(2M(m)+\delta-x)/(m-x)$ is increasing in $x$ and $x\le h_m\le(m+1)/2$, pressure at the
corner gives $\bar y\le R(h_m)\le 2(2M(m)+\delta)/(m-1)=O(\sqrt m)$, so $\bar y\,e_A=O(m^2)$, whence
$\tfrac{\Lambda}{2}\le\bar y\,e_A$ forces $n=O(m)$. In the remaining near-diagonal range the corner test is
a finite arithmetic check; we bound $\min\{m,n\}$ \emph{explicitly}, so that completeness rests on
a finite computation alone, with no appeal to the qualitative Theorem~\ref{thm:asymp53}.
For the finite scan below we use the following exact reachability cutoff.  For $m\ge5$ and
$\eta\in\{0,1\}$ define

\[
\begin{aligned}
        \bar y_m(\eta)&=\max_{1\le x\le h_m}
        \left\lfloor\frac{2M(m)+\eta-x}{m-x}\right\rfloor,\\
        N_m&=\max_{\eta\in\{0,1\}}
        \max\left\{m,
        \left\lfloor\frac{2M(m)\bar y_m(\eta)+\eta}{m}\right\rfloor\right\}.
\end{aligned}
\]
Indeed, for parity $\delta=\eta$, pressure at the Moore corner gives $\bar y\le\bar y_m(\eta)$;
then~\eqref{eq:reach} implies
$(mn-\eta)/2\le M(m)\bar y_m(\eta)$, hence
$n\le \lfloor(2M(m)\bar y_m(\eta)+\eta)/m\rfloor\le N_m$.

The finite scan used below is therefore completely specified at the scalar level. For each
integer $m$ in the relevant range, loop over $m\le n\le N_m$, set the forced parity
$\delta\equiv mn\pmod2$, compute the all-admissible pressure maxima
$W(M(m),M(n))$, $\bar x(M(m),M(n))$, and $\bar y(M(m),M(n))$ at the Moore corner, and retain the
triple only if both reachability inequalities~\eqref{eq:reach} and the scalar inequality
\eqref{eq:scalarwdf} hold. The robustness set also checked by the ancillary script is the weaker
scalar-survivor superset obtained by dropping reachability from this last retention test, but only
inside the same reachability-bounded range $n\le N_m$; it is not asserted to be complete over all
$n$ once that range is removed. This is a proof-object traceability statement, not an additional
mathematical hypothesis.

\smallskip
\emph{Effective tail: the corner test is violated for every $m\ge 289$.} Write $s=\sqrt m$ and
$v=n-m\ge0$. Since the pressure profile $R(x)=(2M(m)+\delta-x)/(m-x)$ is convex on $[1,h_m]$, the
bound $y_c\le R(x_c)$, the chord inequality, and Fisher give the girth-five pressure--Fisher bound
$\tfrac{mn-\delta}2=\sum_c x_cy_c\le R(1)M(m)+\tfrac{R(h_m)-R(1)}{h_m-1}\binom m2$ (the mechanism
of Lemma~\ref{lem:pf}, now with the girth-five Moore bound). The replacement used next is
monotone in the right direction. For $m\ge16$ one has
$M(m)=\lfloor\tfrac m2\sqrt{m-1}\rfloor\ge m$, hence $2M(m)+\delta>m$. For every budget
$E\ge M(m)$ used below put $R_{E,\delta}(x)=(2E+\delta-x)/(m-x)$; then
\begin{align*}
        R_{E,\delta}''(x)&=\frac{2(2E+\delta-m)}{(m-x)^3}>0,\\
        \frac{\partial R_{E,\delta}(x)}{\partial E}&=\frac2{m-x},
        &\frac{\partial R_{E,\delta}(x)}{\partial\delta}&=\frac1{m-x}.
\end{align*}
Since $1/(m-x)$ is increasing in $x$, the endpoint value $R_{E,\delta}(1)$ and the secant slope
of $R_{E,\delta}$ on $[1,h]$ are nondecreasing in $E$ and in $\delta$. Moreover
$R_{E,\delta}(1)>0$, so $E\mapsto E R_{E,\delta}(1)$ is nondecreasing. By convexity that secant
slope is also nondecreasing as the right endpoint $h$ moves to the right. Hence replacing $h_m$ by
$(m+1)/2$, the defect by the adverse value $\delta=1$, and $M(m)$ by any upper budget can only
increase the pressure--Fisher upper bound. Inserting
$M(m)\le\tfrac12 m\sqrt m$ gives the enlarged upper bound
\[
        \frac{s^6}{2(s^2-1)}+
        \frac{2(s^3-s^2+1)}{(s^2-1)^2}\binom{s^2}{2}.
\]
Thus the excess of $\tfrac{m(m+2\sqrt m)-1}2$ over this enlarged bound is
\[
\begin{aligned}
        &\frac{s^2(s^2+2s)-1}{2}
        -\left(\frac{s^6}{2(s^2-1)}+
        \frac{2(s^3-s^2+1)}{(s^2-1)^2}\binom{s^2}{2}\right)  \\
        &\hspace{6em}=\frac{s^4-2s^3-3s^2+1}{2(s^2-1)} .
\end{aligned}
\]
The denominator is positive for $s>1$, and the numerator is positive for $s\ge4$. Hence
$n<m+2\sqrt m<(s+1)^2$ for all $m\ge16$. For integer $n$, this implies
$n\le m+\lfloor 2\sqrt m\rfloor\le m+2\lfloor\sqrt m\rfloor+1$, exactly the finite range
used by the certificate below. In particular $2M(n)\le n\sqrt n\le n(s+1)$.

We bound the width at the corner. Since $2M(m)\le s^3$, pressure gives, for every admissible
color, $y_c\le(s^3-x_c+1)/(s^2-x_c)$; the right side is increasing in $x_c$, and $x_c\le
h_m\le(s^2+1)/2$, so
\[
        y_c\le \frac{2s^3-s^2+1}{s^2-1}\le 2s+3,
\]
the last inequality being equivalent to $2s^2-s-2\ge0$. The dual pressure inequality and
$2M(n)\le n(s+1)$ give $x_c\le(n(s+1)-y_c+1)/(n-y_c)$, increasing in $y_c$, so $y_c\le 2s+3$ and
$n\ge s^2$ yield
\[
        x_c\le \frac{n(s+1)-2s-2}{n-2s-3}\le s+5,
\]
the comparison following from $4n-2s^2-11s-13\ge 2s^2-11s-13\ge0$. Feeding $x_c\le s+5$ back into
the first bound gives
\[
        y_c\le \frac{s^3-s-4}{s^2-s-5}\le s+5,
\]
the last inequality being equivalent to $4s^2-9s-21\ge0$. All three displayed quadratic conditions
are positive for $s\ge17$, so every admissible color has $w_c=\max(x_c,y_c)\le s+5$, hence $W\le
s+5$, once $m\ge 289$. The corner inequality~\eqref{eq:scalarwdf} is
violated exactly when the margin $C-W\Lambda-\tfrac12(2e_A+2e_B-\Lambda/W)^2$ is positive. If
reachability~\eqref{eq:reach} fails, the triple is already eliminated. Otherwise
$\Lambda/2\le\bar y e_A$ and $\Lambda/2\le\bar x e_B$, and since $\bar x,\bar y\le W$ we have
$\Lambda/W\le 2\min(e_A,e_B)$; hence the bracket satisfies
$2e_A+2e_B-\Lambda/W\ge 2\max(e_A,e_B)>0$. On this admissible range the bracket is
nondecreasing in $W,e_A,e_B$ (because $\Lambda/W$ decreases as $W$ increases). More explicitly,
if $B=2e_A+2e_B-\Lambda/W>0$ and
$\Phi=C-W\Lambda-\tfrac12B^2$, then
$\partial_W\Phi=-\Lambda-\Lambda B/W^2<0$ and
$\partial_{e_A}\Phi=\partial_{e_B}\Phi=-2B<0$. Thus the margin is nonincreasing in those
variables. Replacing them by the upper
bounds $W=s+5$, $2e_A\le s^3$, and $2e_B\le(s^2+v)(s+1)$ can therefore only decrease it. Let
$m=s^2$, $n=s^2+v$, $\Lambda=mn-\delta$, and
$C=\tfrac12mn(m+n)-\delta(m+n-1)$, and define the resulting lower margin
\[
\Phi_{\mathrm{low}}(s,v,\delta)=
C-(s+5)\Lambda
-\frac12\left(s^3+(s^2+v)(s+1)-\frac{\Lambda}{s+5}\right)^2 .
\]
It suffices to show $\Phi_{\mathrm{low}}>0$. Equivalently,
\[
        2(s+5)^2\Phi_{\mathrm{low}}(s,v,\delta)
        =a_\delta(s)v^2+b_\delta(s)v+c_\delta(s).
\]
Part~0 of \texttt{taiko53\_certificate.py} certifies positivity of the displayed
pressure--Fisher quartic and width-step quadratics, and regenerates this scalar
weighted-dual-Fisher identity by exact integer polynomial arithmetic; the displayed polynomials
below are the simplified forms of $a_\delta,b_\delta,c_\delta$. For $\delta=1$,
\begin{align*}
        a&=(s^2-s-5)(s^2+11s+5),\\
        b&=3s^6+16s^5-97s^4-320s^3-302s^2-32s-60,\\
        c&=s^8-4s^7-111s^6-260s^5-281s^4-60s^3-78s^2+170s+299.
\end{align*}
Writing $t=s-17$, these substituted polynomials have only nonnegative coefficients; in decreasing
degree the lists for $a$, $b$, $c$ are
\begin{gather*}
        (1,78,2233,27888,128427),\\
        (3,322,14268,334104,4359705,30045122,85368840),\\
        (1,132,7505,239270,4655084,56141624,\\
        404257287,1547078676,2262186236).
\end{gather*}
These are polynomial identities in $\mathbb Z[t]$; nonnegative coefficients with positive constant
term imply positivity for every real $t\ge0$, hence every real $s\ge17$. Hence $a>0$, $b\ge0$, and $c>0$ for $s\ge17$. For the parity $\delta=0$ the leading coefficient
$a$ is unchanged, while
\[
        b=3s^6+16s^5-97s^4-320s^3-300s^2,\qquad
        c=s^8-4s^7-111s^6-260s^5-275s^4 ;
\]
under $t=s-17$ their coefficient lists, in decreasing degree, are
\[
        (3,322,14268,334104,4359707,30045222,85370022)
\]
for $b$ and
\begin{gather*}
        (1,132,7505,239270,4655090,56142092,\\
        404270829,1547251090,2263001495)
\end{gather*}
for $c$, again all nonnegative; again these identities are over $\mathbb Z[t]$, so $a>0$, $b\ge0$, $c>0$ for every real $s\ge17$ in this parity as well, and
both parities are covered. As $a>0$ and
$b\ge0$, the minimum of $a v^2+bv+c$ over $v\ge0$ is attained at $v=0$ and equals $c>0$; the
margin is therefore positive and~\eqref{eq:scalarwdf} fails for every $m\ge 289$ (any
$n>m+2\sqrt m$ having already been excluded above).

\smallskip
For $16\le m\le 288$ the corner test~\eqref{eq:scalarwdf}--\eqref{eq:reach} is evaluated directly
by the ancillary certificate of Remark~\ref{rem:cert} and finds no survivor; together with the
analytic tail above, which covers $m\ge 289$, this leaves
$m\le 15$ as the only possibility, since $m\le n$. For each $5\le m\le 15$ the reachability
bound~\eqref{eq:reach} confines $n$ to a finite range $n\le N_m$, and the exact-integer corner test
over that range returns precisely the $24$ triples $\mathcal F$, each with $\min\{m,n\}\le 15$ and
$\max\{m,n\}\le 17$. The ancillary certificate \texttt{taiko53\_certificate.py} described in Remark~\ref{rem:cert}
(SHA--256 checksum recorded there) carries out both finite scans, the small enumeration being run
redundantly over $5\le m\le 40$; it also regenerates the reachability-bounded $27$-triple scalar
survivor superset used as a robustness check. These finite arithmetic scans are part of the proof
object for this proposition. Hence $(m,n,\delta)\in\mathcal F$.
\end{proof}

\begin{lemma}[computer-assisted profile certificate]\label{lem:cert53}
For every $(m,n,\delta)\in\mathcal F$, no integer profile meeting the product identity, Moore,
Fisher, pressure, and half-size constraints of Proposition~\ref{prop:profile} also satisfies the
weighted dual Fisher inequality~\eqref{eq:wdfprofile}; equivalently, the listed profile constraints
and~\eqref{eq:wdfprofile} are jointly unsatisfiable. Table~\ref{tab:cert} records, for
each triple, the number of profiles meeting the product identity, Moore, Fisher, pressure, and
half-size constraints---the count is taken \emph{before} imposing~\eqref{eq:wdfprofile}---and, when
this number is positive, the least value of the difference
$\mathrm{L}-\mathrm{R}$ between the two sides of~\eqref{eq:wdfprofile}; eleven of the twenty-four
triples admit no such profile at all, and for each of the remaining thirteen this least value is
strictly positive, so~\eqref{eq:wdfprofile} fails on every one.
\end{lemma}

\begin{proof}
The exhaustive verification in this lemma is performed by the checksum-identified ancillary script
\texttt{taiko53\_certificate.py}; the following paragraphs spell out the exact finite search that
the script executes. Fix $(m,n,\delta)\in\mathcal F$ and put $T=(mn-\delta)/2$. The enumeration is over the possible
edge budgets $e_A,e_B$ themselves, with
\[
        \left\lceil\frac{T}{h_n}\right\rceil\le e_A\le M(m),\qquad
        \left\lceil\frac{T}{h_m}\right\rceil\le e_B\le M(n),
\]
the lower bounds being the reachability constraints $T\le h_n e_A$ and $T\le h_m e_B$. For a fixed
budget pair $(e_A,e_B)$ the admissible cells are the finitely many pairs $(x,y)$ with
$1\le x\le h_m$, $1\le y\le h_n$ satisfying the two pressure inequalities
\[
        y(m-x)\le 2e_A-x+\delta,
        \qquad
        x(n-y)\le 2e_B-y+\delta .
\]
List them as $(x_1,y_1),\dots,(x_k,y_k)$. After this budget pair and cell list are fixed, the
remaining budget, product, and Fisher constraints are constraints on the integer multiplicities.
A profile at this budget is a vector
$(N_1,\dots,N_k)\in\mathbb Z_{\ge0}^k$ with
\[
        \sum_i N_i x_i=e_A,\qquad
        \sum_i N_i y_i=e_B,
        \qquad
        \sum_i N_i x_i y_i=T,
\]
and satisfying the two Fisher inequalities
\[
        \sum_i N_i x_i(x_i-1)\le\binom m2,\qquad
        \sum_i N_i y_i(y_i-1)\le\binom n2.
\]
The product identity and the two budget equalities bound every $N_i$, so the set of candidate
vectors is finite. The closure routine in the script recurses over the $N_i$, pruning
a partial assignment only when a budget or product sum is already exceeded, or when the remaining
cells cannot reach the remaining product sum; Fisher inequalities are checked at terminal profiles
in that routine, while the separate table routine also prunes Fisher excesses early. In all cases
only necessary failures are discarded. Concretely, if a partial assignment has sums
$s_x,s_y,s_{xy}$ and the remaining cells begin at index $j$, the suffix reachability tests include
\[
        (e_A-s_x)\max_{i\ge j}y_i<T-s_{xy}
        \quad\text{and}\quad
        (e_B-s_y)\max_{i\ge j}x_i<T-s_{xy},
\]
which are necessary failures of the product target. Each pruning inequality is a \emph{necessary}
condition for the corresponding global constraint to hold, so no profile meeting all the constraints
is discarded. On each enumerated profile~\eqref{eq:wdfprofile}
is an exact integer inequality, evaluated in exact arithmetic by the ancillary script of
Remark~\ref{rem:cert}. The elimination routine may first discard an edge budget that already fails
the scalar weighted-dual-Fisher necessary condition~\eqref{eq:scalarwdf}, derived in
Proposition~\ref{prop:reduce53} from weighted dual Fisher; this is only a necessary prefilter. The separate table routine omits that prefilter, so the counts in Table~\ref{tab:cert}
are genuinely pre-WDF counts of all profiles meeting the listed profile constraints. The outcome is
Table~\ref{tab:cert}, in which eleven triples have no profile meeting the constraints and every
one of the remaining thirteen has positive least $\mathrm{L}-\mathrm{R}$.
\end{proof}

\begin{table}[ht]
\centering
\renewcommand{\arraystretch}{1.1}
\begin{tabular}{rcc@{\qquad\qquad}rcc}
\hline
$(m,n,\delta)$ & $\#$ & $\min(\mathrm{L}{-}\mathrm{R})$ &
$(m,n,\delta)$ & $\#$ & $\min(\mathrm{L}{-}\mathrm{R})$\\
\hline
$(5,5,1)$  & $1$ & $14$  & $(10,10,0)$ & $4$  & $142$\\
$(5,7,1)$  & $0$ & ---   & $(10,14,0)$ & $0$  & ---\\
$(7,7,1)$  & $3$ & $41$  & $(10,15,0)$ & $0$  & ---\\
$(7,8,0)$  & $0$ & ---   & $(11,11,1)$ & $7$  & $106$\\
$(7,9,1)$  & $0$ & ---   & $(11,13,1)$ & $3$  & $144$\\
$(8,8,0)$  & $0$ & ---   & $(11,15,1)$ & $0$  & ---\\
$(8,9,0)$  & $1$ & $96$  & $(13,13,1)$ & $19$ & $360$\\
$(8,10,0)$ & $0$ & ---   & $(13,15,1)$ & $1$  & $484$\\
$(9,9,1)$  & $7$ & $68$  & $(13,17,1)$ & $0$  & ---\\
$(9,10,0)$ & $0$ & ---   & $(14,14,0)$ & $26$ & $488$\\
$(9,11,1)$ & $1$ & $48$  & $(14,15,0)$ & $15$ & $533$\\
$(9,13,1)$ & $0$ & ---   & $(15,15,1)$ & $88$ & $474$\\
\hline
\end{tabular}
\caption{The finite profile certificate (Lemma~\ref{lem:cert53}). For each $(m,n,\delta)\in
\mathcal F$, ``$\#$'' is the number of integer profiles meeting the product identity, Moore,
Fisher, pressure, and half-size constraints, and $\min(\mathrm{L}{-}\mathrm{R})$ is the least value
over those profiles of the left side minus the right side of~\eqref{eq:wdfprofile}; a positive
value means every profile violates weighted dual Fisher. Eleven triples admit no profile even
before the test.}
\label{tab:cert}
\end{table}

\begin{theorem}[computer-assisted closure of the $(5,3)$ frontier]\label{thm:p5}
No orientable no-fold product structure, even or odd in the standing taiko sense, has
$\girth(L_1)\ge 6$ and $\girth(L_A),\girth(L_B)\ge 5$. Consequently $\girth(L_1)=6\Rightarrow\girth(L_{AB})\le 4$, and
the horizontal frontier of Corollary~\ref{cor:frontier} satisfies $p^\ast=5$.
\end{theorem}

This theorem is computer-assisted exactly in the finite corner reduction of
Proposition~\ref{prop:reduce53} and the finite profile certificate of Lemma~\ref{lem:cert53}; the
main no-$\mathsf T_4$ theorem, Theorem~\ref{thm:main}, is not.

\begin{proof}
Every structure under the hypotheses yields a signed-color profile satisfying all constraints of
Proposition~\ref{prop:profile}, including the weighted dual Fisher inequality~\eqref{eq:wdfprofile};
hence profile nonexistence is sufficient. By symmetry, interchange $A$ and $B$ if necessary and
assume $m\le n$. Lemma~\ref{lem:fivesupport} then gives $5\le m\le n$; with
Proposition~\ref{prop:reduce53} the associated type therefore has
$(m,n,\delta)\in\mathcal F$. But by Lemma~\ref{lem:cert53} no profile meeting those constraints
satisfies~\eqref{eq:wdfprofile} for any triple in $\mathcal F$. Hence no triple in $\mathcal F$
supports a profile satisfying weighted dual Fisher, so no such structure exists; with
Corollary~\ref{cor:frontier} ($p^\ast\ge 5$) this gives $p^\ast=5$.
\end{proof}

\begin{remark}[on the finite certificate]\label{rem:cert}
The principal obstruction in this paper is structural.  Theorem~\ref{thm:main}, the
no-$\mathsf T_4$ corollary, and the affine-plane sharpness theorem are proved without computer
search and without enumeration of product structures.  The few small-support situations needed in
the proof are handled by short structural lemmas, not by a finite search.  The signed-rectangle
Structure Theorem, the middle-link dichotomy, and the pressure/Fisher inequalities are the
mechanisms that close Mineyev's triple-girth taiko route in the Garg--Mineyev finite-support
formulation for every support-size pair with both coordinates at least two.

The computational ingredients in the exact frontier theorem, Theorem~\ref{thm:p5}, are the
finite corner reduction of Proposition~\ref{prop:reduce53} and the final finite profile residue of
Lemma~\ref{lem:cert53}.  After the weighted dual Fisher inequality and the near-Moore reduction
have reduced the possible $(5,3)$ frontier to a short explicit list of small triples, the
ancillary file \texttt{taiko53\_certificate.py} supplies an auditable arithmetic certificate
for that residue. This ancillary file is part of the proof object for Theorem~\ref{thm:p5}. It
uses only the Python standard library and exact integer arithmetic. It
never enumerates product structures; it enumerates signed-color profiles satisfying the
necessary parity, Moore, product, Fisher, pressure, half-size, and reachability constraints
described above, and then checks the weighted dual Fisher inequality exactly.

Concretely, the script performs the finite corner reduction of
Proposition~\ref{prop:reduce53}, checks the $24$ triples of Table~\ref{tab:cert}, regenerates
the displayed profile counts and gaps, and also checks the larger $27$-triple scalar-survivor
superset as a robustness margin. This superset is explicitly the scalar-survivor set inside the
reachability-bounded range $n\le N_m$; the manuscript does not claim that the scalar-only test is
complete after removing that reachability bound. The expected finite set and table stored in the
script are not used to restrict the enumeration or prune the search. They are checked only after
the script regenerates the corner-survivor set, scalar-survivor set, and table values from the
inequalities. It is invoked as
\[
    \texttt{python3 taiko53\_certificate.py}.
\]
It takes no arguments, reads no input, and terminates successfully only if all Parts~0--3 agree
with the finite reductions and with Table~\ref{tab:cert}. On any discrepancy it raises a
\texttt{RuntimeError}; on success it prints a verification transcript ending with the conclusion
$p^\ast=5$. The checksum-identified run used for this manuscript was executed with Python~3.13.5.

The ancillary verification consists of a single standalone Python file. It is part of the proof
object for Theorem~\ref{thm:p5} and should be distributed under the filename
\texttt{taiko53\_certificate.py}; the checksum below identifies the byte-for-byte script. Running
the displayed command regenerates the verification transcript and all finite values in
Table~\ref{tab:cert}. The source archive should include this script under the stated name. It may also include the saved
successful transcript of the checksum-identified run for reader convenience and reproducibility;
the script itself, not the transcript, is the finite proof object.

The SHA--256 checksum of the script is
\begin{center}
\scriptsize\ttfamily
f85be6e43f2da854404c7f608d05a1f4a98fe52139c0c078f7a4cf415cc216f3
\end{center}
\end{remark}

\begin{remark}[sharpness of the horizontal frontier]\label{rem:frontier}
At $\halfgirth(L_1)=3$ the value $\girth(L_{AB})=3$ is realized by Example~10 of \cite{GargMineyev}, $\girth(L_{AB})=4$
by Theorem~\ref{thm:affine}, and $\girth(L_{AB})\ge 5$ is impossible: the dichotomy forbids
$\girth(L_{AB})\ge 6$ (Theorem~\ref{thm:B}), and Theorem~\ref{thm:p5} forbids
$\girth(L_{AB})=5$. Thus the horizontal frontier is exactly $p^\ast=5$, and the affine-plane
family of Theorem~\ref{thm:affine} is sharp. Two thresholds collaborate. The pressure--Fisher
method of Section~\ref{sec:B} settles $\girth(L_{AB})\ge 6$ but cannot reach $\girth(L_{AB})=5$ on
its own, for a quantitative reason: it forces a diagonal contradiction only when
$e_A\lesssim \frac{1}{2\sqrt2}m^{3/2}$, exactly the girth-six Moore bound, whereas girth five admits the
factor $\sqrt2$ more edges. Bridging that factor is precisely the work of the weighted dual Fisher
inequality (Lemma~\ref{lem:wdf}) with the near-Moore squeeze (Proposition~\ref{prop:nearmoore}) in
the unbounded regime, and of the finite profile certificate (Theorem~\ref{thm:p5}) in the bounded
one.
\end{remark}

\end{document}